\documentclass[11pt]{article}
\usepackage{amsmath,amsfonts,amssymb,amsthm,fancyhdr,bm,mathtools,enumitem,bbm}
\usepackage[hyphens]{url}
\usepackage[hidelinks]{hyperref}
\usepackage{fullpage}
\usepackage[dvipsnames]{xcolor}
\usepackage[capitalize]{cleveref}
\usepackage[color=Purple!50!white,textsize=tiny]{todonotes}
\usepackage[numbers,sort]{natbib}
\usepackage{mathtools}
\usepackage{tikz,ifthen}
\usepackage{comment}
\usepackage{dsfont}
\usepackage{soul}

\usetikzlibrary{decorations.markings,arrows.meta}
\tikzset{vert/.style={draw, fill=black, circle, inner sep=2pt}}

\newtheorem{theorem}{Theorem}[section]
\newtheorem{lemma}[theorem]{Lemma}
\newtheorem{proposition}[theorem]{Proposition}

\newtheorem{corollary}[theorem]{Corollary}

\theoremstyle{definition}
\newtheorem{definition}[theorem]{Definition}

\newtheorem{claim}{Claim}
\makeatletter
\@addtoreset{claim}{section}
\makeatother

\crefname{equation}{equation}{equations}
\crefname{lemma}{Lemma}{Lemmas}
\crefname{proposition}{Proposition}{Propositions}
\crefname{claim}{Claim}{Claims}
\crefname{theorem}{Theorem}{Theorems}
\crefname{conjecture}{Conjecture}{Conjectures}
\crefname{figure}{Figure}{Figures}

\AddToHook{env/lemma/begin}{\crefalias{theorem}{lemma}}
\AddToHook{env/proposition/begin}{\crefalias{theorem}{proposition}}
\AddToHook{env/corollary/begin}{\crefalias{theorem}{corollary}}
\AddToHook{env/conjecture/begin}{\crefalias{theorem}{conjecture}}
\AddToHook{env/claim/begin}{\crefalias{theorem}{claim}}
\AddToHook{env/question/begin}{\crefalias{theorem}{question}} 
\AddToHook{env/definition/begin}{\crefalias{theorem}{definition}} 
\AddToHook{env/remark/begin}{\crefalias{theorem}{remark}} 

\newlist{lemenum}{enumerate}{1}
\setlist[lemenum]{label=(\alph*), ref=\thelemma(\alph*)}
\crefalias{lemenumi}{lemma}

\let\leq\leqslant
\let\geq\geqslant

\title{Thinning and sprinkling: from robust sampling to almost Hamiltonicity}
\author{Micha Christoph \thanks{Department of Mathematics, ETH Z\"{u}rich, 8092 Z\"urich, Switzerland. Email: \texttt{micha.christoph@math.ethz.ch}. Research supported by SNSF Ambizione Grant No. 216071.} \and Zach Hunter\thanks{Department of Mathematics, ETH Z\"{u}rich, 8092 Z\"urich, Switzerland. Email: \texttt{zach.hunter@math.ethz.ch}.} \and Benny Sudakov \thanks{Department of Mathematics, ETH Z\"{u}rich, 8092 Z\"urich, Switzerland. Email: \texttt{benjamin.sudakov@math.ethz.ch}. Research supported in part by SNSF Grant 
200021-228014.}}
\date{}

\begin{document}

\maketitle
\begin{abstract}
    We develop the thinning--sprinkling technique, a general method for proving robustness of graph properties under random vertex sampling. Using it, we show that random induced subgraphs of tough graphs, high-degree connected vertex-transitive graphs, and nearly regular sublinear expanders retain strong connectivity or expansion properties with very high probability. We also prove that every $k$-connected graph with $k=\omega(\log n)$ contains a spanning bipartite subgraph that is $\Omega(k)$-connected.
    
    Using these robustness results, we further develop a general framework for constructing almost Hamilton cycles from randomly sampled highly connected subgraphs. As a consequence, we show that tough graphs, connected vertex-transitive graphs and nearly regular expanders contain a cycle of length at least $(1-o(1))n$ whenever the toughness or degree is polylogarithmically large. This gives asymptotic solutions of longstanding conjectures of Chv\'atal and Lov\'asz on Hamiltonicity of tough and vertex-transitive graphs.
\end{abstract}
\section{Introduction}
Passing from a mathematical structure to a random subset of its elements, a procedure called subsampling, is a fundamental idea in mathematics. The resulting random substructure can often reveal useful information about the original one. For example, random Fourier sampling preserves the geometry of sparse signals and underlies compressed sensing, while subsampling Cayley graphs has been used to transfer analytic information and study structural questions in geometric and measured group theory. More generally, if a property persists under subsampling, this indicates that it is present in the original structure in a robust form. This perspective is central to random graph theory, and more recently there has been growing interest in the robustness of classical results in graph theory under subsampling.

Complementing the robustness perspective, subsampling is itself a powerful tool. A prominent example is graph sparsification, where random edge sampling produces a much sparser graph while approximately preserving connectivity and expansion properties of the original one. Another is property testing, where one seeks to determine efficiently whether a structure has a given property or is far from having it, often by inspecting only a small random sample.

The results presented in this paper illustrate both of these aspects of subsampling. We develop the \emph{thinning--sprinkling technique}, a method for showing that certain graph properties are robust under random vertex subsampling. We then develop a framework that uses such robustness results as a black-box to obtain long cycles in several well-studied classes of graphs, leading to progress toward longstanding conjectures of Chv\'atal and Lov\'asz on Hamiltonicity.

\subsection{The thinning-sprinkling technique}
In this subsection, we describe the basic idea behind the thinning--sprinkling technique and state several consequences of it. Although we believe these results are interesting in their own right, our main emphasis is on the method rather than on any individual application. The results below are intended to illustrate the flexibility and strength of the approach. We believe that thinning--sprinkling provides a useful general framework for studying robustness under random sampling and may have applications beyond those developed here.

We begin with a brief informal description of the thinning--sprinkling technique, leaving the technical details to Section~\ref{sec: connectivity}. Let $G$ be a graph, let $p\in[0,1]$, and let $V_p\subseteq V(G)$ be a random subset containing every vertex independently with probability $p$. Suppose that, given a property $\mathcal P$, we want to show that $G[V_p]$ is likely to satisfy $\mathcal P$. We consider nested random sets $V_{p/4}\subseteq V_{p/2}\subseteq V_p$ from two complementary perspectives. Here $V_{p/2}$ is obtained by subsampling $V_p$ with probability $1/2$, and $V_{p/4}$ is obtained by subsampling $V_{p/2}$ once more with probability $1/2$.

The thinning step consists of finding an auxiliary property $\mathcal P'$ of $V_{p/4}$ and $V_{p/2}$ with two features. First, $\mathcal P'$ should be likely to hold conditional on $G[V_p]$ not satisfying $\mathcal P$. Second, $\mathcal P'$ should be easier to control than $\mathcal P$. To understand $\mathcal P'$, we switch to the sprinkling perspective. Once $V_{p/4}$ is revealed, the additional vertices in $V_{p/2}\setminus V_{p/4}$ may be viewed as adding an independent random set of vertices, each included with probability roughly $p/4$.
We then show that this sprinkling is very likely to interact with $V_{p/4}$ in a way that prevents $\mathcal P'$ from holding. Thus, thinning gives a lower bound on the probability of $\mathcal P'$, while sprinkling gives a conflicting upper bound, forcing $G[V_p]$ to satisfy $\mathcal P$ with the desired probability.
In principle, the thinning--sprinkling technique can be applied to a wide range of properties $\mathcal P$. We have found it particularly useful when $\mathcal P$ is a global property expressing a strong form of connectivity, such as expansion, and $\mathcal P'$ implies, among other things, that $G[V_{p/2}]$ is disconnected.

The combination of thinning and sprinkling is not entirely new, but previous applications have largely focused on specific graph classes such as hypercubes and expanders. A first version of it appeared in work of Diskin and Krivelevich in \cite{Sahar} and it was developed much further in \cite{mixing_time} by Anastos, Diskin, Lichev and Zhukovskii to resolve well-known conjectures about the typical diameter and mixing time of the giant component of the hypercube under bond percolation. Both of these results focused on edge subsampling. Recently, Christoph, Diskin, Lichev and Sudakov \cite{polytope} transferred the method to vertex subsampling on the hypercube in their work on the expansion of random $0/1$-polytopes. 

A related line of work, originating in theoretical computer science, concerns the preservation of connectivity under random sampling. In 1994, Karger~\cite{karger} gave a beautiful randomized algorithm which implies that, in a $k$-edge-connected graph on $n$ vertices, the number of edge cuts of size $m$ is at most $n^{2m/k}$. A simple union bound then shows, under suitable quantitative assumptions, that edge-connectivity is preserved under edge subsampling. 
There is no known equivalent of Karger's algorithm for vertex cuts. This makes preserving vertex-connectivity in random induced subgraphs substantially more challenging. Censor-Hillel, Ghaffari, Giakkoupis, Haeupler and Kuhn~\cite{connected_subsampling} nevertheless proved that vertex-connectivity is robust under vertex subsampling. Their argument shares some of the general flavor of our approach, but the details are rather different from the thinning--sprinkling method developed here. We give a new proof of their result which also serves as a detailed illustration of our method.

\begin{theorem}\label{thm: subsampling k-connected}
Let $G$ be a $k$-connected graph and $p\in[0,1]$ such that $p^2k=\omega(\log n)$. Then, $G_p$ is $\Omega(p^2 k)$-connected with probability at least $1-2^{-\Omega(p^2k)}$.
\end{theorem}

A natural refinement of vertex-connectivity is toughness, a well-studied parameter which controls not only whether a vertex set disconnects a graph, but also how many components its removal can create. A graph $G$ is $t$-tough if, for every $S\subseteq V(G)$ such that $G-S$ has $k\geq2$ components, we have $|S|\geq tk$. By convention, complete graphs have infinite toughness. As an extension of the connectivity result above, we show that this stronger form of connectivity is itself robust under random vertex subsampling.
\begin{theorem}\label{thm: subsample tough}
     Let $G$ be a $t$-tough graph and $p\in[0,1]$ such that $p^2 t=\omega(\log n)$. Then, $G[V_p]$ is $\Omega( p^2t)$-tough with probability at least $1-2^{-\Omega( p^2t)}$.
\end{theorem}

Another central class is that of vertex-transitive graphs. A graph is vertex-transitive (transitive for short) if for any two vertices there is an automorphism mapping one to the other. Such graphs play an important role in graph theory, group theory, and probability: Cayley graphs form a fundamental family of examples, and transitive graphs also arise naturally in the study of random walks and percolation. The symmetry of transitive graphs forces strong connectivity properties. In particular, a classical result of Mader \cite{mader} implies that every connected $d$-regular vertex-transitive graph is $d$-edge-connected, and consequently $1$-tough. Our next result shows that this toughness is essentially preserved under subsampling.
\begin{theorem}\label{thm: vertex-transitive subsampling}
    The following holds for every absolute constant $\varepsilon>0$. Let $G$ be a connected $d$-regular vertex-transitive graph on
    $n$ vertices and let $p\in[0,1]$ satisfy $p^2d=\omega(\log n)$. Then $G_p$ is $(1-\varepsilon)$-tough with probability at least
    $1-2^{-\Omega(p^2d)}$.
\end{theorem}

Our third result concerns sublinear expanders. Introduced by Koml\'os and Szemer\'edi~\cite{komlos-szemeredi} in their work on clique subdivisions, sublinear expanders have developed into a powerful tool in extremal graph theory, with applications to a wide range of problems; see the surveys~\cite{letzter-survey,montgomery-survey} for further examples. There are several notions of sublinear expansion in the literature, tailored to different applications. We work with the following particularly simple one.
\begin{definition}
We say that $G$ is a $\lambda$-expander if for every $U\subseteq V(G)$ with $|U|\leq |V(G)|/2$ we have
$|N_G(U)|\geq \lambda |U|$,
where $N_G(U)$ denotes the external neighborhood of $U$.
\end{definition}

Random subsampling is frequently used in arguments involving sublinear expanders: one passes to a random set of vertices and then proves, usually by a problem-specific argument, that the particular properties needed later survive. For some spectral notions of expansion, general subsampling results are already known~\cite{spectral}. We show that, for nearly regular graphs, expansion itself is preserved under vertex subsampling. This allows the resulting random induced subgraph to be treated directly as another expander, avoiding a separate analysis in each application; we will use this perspective later in our study of vertex-transitive graphs.

\begin{theorem}\label{thm: regular expander subsampling}
Let $\lambda\leq 1$, let $G$ be a $\lambda$-expander with minimum degree $d$ and maximum degree $(1+o(1))d$. Let $p\in[0,1]$ be such that $p^2\lambda d=\omega(\log n)$. Then, $G[V_p]$ is an $\Omega(p\lambda)$-expander with probability at least $1-2^{-\Omega(p^2\lambda d)}$.
\end{theorem}

For the final application of the thinning--sprinkling technique, we return to the connectivity of graphs. An old conjecture of Thomassen \cite{thomassen} from 1989 asks whether there exists a function $f$ such that, for every $k$, every $f(k)$-connected graph has a spanning bipartite subgraph which is $k$-connected. More recently, there has been a sequence of papers \cite{delcourt, yuster, Gutin} making progress towards this question. Our result improves on all of these results and implies that $f$ can be taken to be linear as long as $k=\omega(\log n)$.
\begin{theorem}\label{thm: bipartite}
If $G$ is a $k$-connected graph with $k=\omega(\log n)$, then $G$ contains a spanning bipartite subgraph which is $\Omega(k)$-connected.
\end{theorem}

\subsection{Almost Hamilton cycles}
Using the robustness results obtained through the thinning--sprinkling technique, we further develop a general framework for showing that certain graphs contain almost spanning cycles. Roughly speaking, we show that if a graph is robustly tough and allows a suitable way of finding a spanning linear forest then it has an almost spanning cycle. We apply this principle to tough graphs, vertex-transitive graphs, and expanders. For the first two cases, the resulting almost-spanning cycle results are new. For expanders, the corresponding result was already known, but our argument is significantly simpler and substantially improves the required degree threshold.

The celebrated question concerning the relationship between toughness and Hamiltonicity dates back to a paper of Chv\'atal from 1973~\cite{chvatal}, where he asked whether there exists a constant $t_0$ such that every $t_0$-tough graph is Hamiltonian. Progress on this problem has been limited: for more than 25 years (see \cite{toughlowerbound}), the best general bounds have remained $9/4\leq t_0\leq (1+o(1))\sqrt{n/2}$. Recently, Segal and Verstraete~\cite{segal} showed that every $15$-tough graph contains a cycle of polynomial length with a small exponent. We show that polylogarithmic toughness already suffices for an almost spanning cycle.
\begin{theorem}\label{thm: tough cycle}
    Let $G$ be a $t$-tough graph with $t= \omega(\log^3 n)$. Then, $G$ contains a cycle of length $(1-o(1))n$.
\end{theorem}

Another central open problem concerning Hamiltonicity is the Lov\'asz conjecture for connected vertex-transitive graphs. The strongest natural statement, that every connected vertex-transitive graph has a Hamilton cycle, is false. Lov\'asz~\cite{lovasz} asked whether there exists a connected vertex-transitive graph without a Hamilton path, while a stronger conjecture, commonly attributed to Thomassen, asserts that all but finitely many connected vertex-transitive graphs have a Hamilton cycle; see~\cite{babai}. Two main directions have emerged in the study of these problems. The first asks how long a cycle one can guarantee in every connected vertex-transitive graph. This goes back to Babai~\cite{babai}, who proved a lower bound of $\Omega(\sqrt n)$; the current best bound is $n^{2/3-o(1)}$, due to Buci\'c, Christoph, Pokrovskiy and Steiner~\cite{bucic-christoph-pokrovskiy-steiner}. The second asks under what additional assumptions Hamiltonicity can be guaranteed. Christofides, Hladk\'y and M\'ath\'e~\cite{christofides-hladky-mathe} proved this for connected vertex-transitive graphs of linear degree, and more recently Bedert, Dragani\'c, M\"uyesser and Pavez-Sign\'e~\cite{bedert-draganic-muyesser-pavez-signe} showed that there exists $\varepsilon>0$ such that every sufficiently large connected Cayley graph of degree at least $n^{1-\varepsilon}$ is Hamiltonian. Our result connects these two directions and shows that already for connected vertex-transitive graphs of polylogarithmic degree, there is a cycle covering all but $o(n)$ vertices.
\begin{theorem}\label{thm: vertex-transitive cycle}
    Let $G$ be a connected $d$-regular vertex-transitive graph with $d=\omega(\log^3 n\, \log\log n)$. Then, $G$ contains a cycle of length $(1-o(1))n$.
\end{theorem}

As a final illustration of the subsampling technique and the generality of our approach to long cycles, we apply it to nearly regular sublinear expanders. Letzter, Methuku and Sudakov~\cite{letzter} previously showed that such graphs contain almost spanning cycles under substantially stronger quantitative assumptions. For the standard sublinear-expansion scale $\lambda=1/\log^2 n$, their result requires degree roughly $(\log n)^{70}$. Our approach gives a significantly simpler proof with better quantitative dependence.

\begin{theorem}\label{thm: expander cycle}
Let $G$ be a $\lambda$-expander with minimum degree $d$ and maximum degree $(1+o(1))d$ such that $\lambda d=\omega(\log^3 n)$ and $d=\omega(\log^3 n\log\log n)$. Then, $G$ contains a cycle of length $(1-o(1))n$.
\end{theorem}

\noindent
For $\lambda=1/\log^2 n$, our theorem gives an almost spanning cycle already when $d=\omega(\log^5 n)$. On the other hand, Chen, Liu, Wei and Yang~\cite{chen-liu-wei-yang} recently constructed $d$-regular sublinear expanders with $d=(1/2+o(1))\log^2 n$ whose longest cycle covers only an arbitrarily small fraction of the vertices, leaving a gap between exponents $2$ and $5$.

\section{Framework and notation}\label{sec: notation}
Throughout the paper, $n$ denotes the number of vertices of the graph $G$ and we assume $n$ to be large. Asymptotic notation like $\omega(\cdot)$ is always with respect to $n$. Given a finite set $V$ and $p\in[0,1]$, we say that $U\subseteq V$ is
\emph{$p$-random} if every element of $V$ belongs to $U$ independently with
probability $p$. For a graph $G$, we write $V_p$ for a $p$-random subset of $V(G)$ and
$G_p:=G[V_p]$.

We begin with a brief description of the main idea underlying the thinning--sprinkling technique.
It is based on viewing the same nested random
sets
\[
    V_{p/4}\subseteq V_{p/2}\subseteq V_p
\]
in two complementary ways. Suppose that our goal is to show that $G[V_p]$
satisfies some property $\mathcal P$ (e.g., connectivity) with very high probability. From the
\emph{thinning} viewpoint, conditioned on $V_p$, the set $V_{p/2}$ is a
$1/2$-random subset of $V_p$, and, conditioned on $V_{p/2}$, the set
$V_{p/4}$ is a $1/2$-random subset of $V_{p/2}$. Typically, if $\mathcal P$
fails with some probability $q$, thinning allows us to deduce that
a simpler obstruction $\mathcal E$, depending only on $V_{p/4}$ and
$V_{p/2}$, occurs with probability at least $q'=q'(q)$.

We then estimate the probability of $\mathcal E$ from the complementary
\emph{sprinkling} viewpoint. Take $p'$ to be the solution to
\[
    (1-p/4)(1-p')=1-p/2.
\]
We may couple a $p'$-random set $V_{p'}$ with $V_{p/4}$ so that the two sets
are independent and
\[
    V_{p/2}=V_{p/4}\cup V_{p'}.
\]
Thus, after revealing $V_{p/4}$, we can regard $V_{p'}$ as fresh independent
randomness which is sprinkled onto it. The structure of the ambient graph
then typically gives many opportunities for this sprinkling to destroy the
obstruction $\mathcal E$, allowing us to show that
$\mathbb P(\mathcal E)<q'$. This contradicts the lower bound obtained by
thinning and proves that $\mathcal P$ fails with  probability less than $q$.

In our applications, the sprinkling set $V_{p'}$ will often be exposed in
many independent stages. We introduce some notation for this setup. Let
$U_0,\ldots,U_\ell$ be independent random subsets of $V(G)$, where $U_0$
should be thought of as the initially exposed set $V_{p/4}$ and
$U_1,\ldots,U_\ell$ as successive sprinklings. Set
\[
    W_i=\bigcup_{j=0}^i U_j
    \qquad\text{and}\qquad
    G_i=G[W_i].
\]
For $v\in V(G)$, let $C_i(v)$ denote the component of $v$ in $G_i$, with
$C_i(v)=\emptyset$ if $v\notin W_i$. Additionally, we say that a connected component of $G_i$ is \emph{seeded} if it meets (i.e. has nonempty intersection with) $U_0$.

An important feature of this exploration is that $C_i(v)$ can be determined
without exposing all of $U_0,\ldots,U_i$: it is enough to reveal the
interaction of $U_0,\ldots,U_i$ with
\[
    \{v\}\cup C_i(v)\cup N_G(C_i(v)).
\]
For vertices $v_1,\ldots,v_k$, we write
$\mathcal F_i(v_1,\ldots,v_k)$ for the information revealed by carrying out
these explorations simultaneously for $v_1,\ldots,v_k$ up to stage $i$.
This allows us to expose the random sets adaptively while leaving the
unexplored vertices independent for later sprinkling steps.

We will repeatedly use the following elementary consequence of this setup.
Suppose that $\mathcal A_1,\ldots,\mathcal A_\ell$ are events such that
$\mathcal A_i$ is $\mathcal F_i(v_1,\ldots,v_k)$-measurable and
\[
    \mathbb P\bigl(\mathcal A_i
        \mid \mathcal F_{i-1}(v_1,\ldots,v_k)\bigr)\geq \rho
\]
for every $1\leq i\leq\ell$. Then
\[
    \sum_{i=1}^{\ell}\mathbf 1_{\mathcal A_i}
\]
stochastically dominates $\operatorname{Bin}(\ell,\rho)$. In particular,
standard binomial lower tail bounds remain available even though the successive
events $\mathcal A_i$ need not be independent.

A recurring goal in the sprinkling arguments will be to show that, at many
stages, one of the components containing vertices of $U_0$ must grow. A
minor difficulty is that a component may grow by merging a large component with a small one, so an arbitrary starting vertex need not witness substantial
multiplicative growth during the merge. The following deterministic observation lets us
choose the starting vertex so that every genuine growth event at least
doubles the number of vertices of $U_0$ in its component. Consequently,
such growth can occur at most $\log_2 n$ times. This will allow us to turn
the probabilistic occurrence of many successful sprinkling stages into a
contradiction.

\begin{lemma}\label{lem: minimal vertex}
    Let $G$ be a graph and let $U_0,\ldots,U_\ell\subseteq V(G)$. Let $C$ be
    a seeded component of $G$. Then, there exists
    $v\in C\cap U_0$ such that, for every $0\leq i\leq\ell-1$, either
    \[
    C_i(v)\cap U_0=C_{i+1}(v)\cap U_0
    \qquad\text{or}\qquad
    2|C_i(v)\cap U_0|
        \leq |C_{i+1}(v)\cap U_0|.\]
\end{lemma}

\begin{proof}
    We prove the statement by reverse induction. More precisely, for
    $0\leq j\leq\ell$, we show that there exists $v\in C\cap U_0$ for which
    the conclusion holds for every $j\leq i\leq\ell-1$. For $j=\ell$ there
    is nothing to prove.

    Suppose the assertion holds for some $j\geq1$, witnessed by
    $v\in C\cap U_0$. If
    \[
        C_{j-1}(v)\cap U_0=C_j(v)\cap U_0,
    \]
    then the same vertex $v$ also witnesses the assertion for $j-1$.
    Otherwise, let $D_1,\ldots,D_h$ be the components of $G_{j-1}$ which
    meet $U_0$ and are contained in $C_j(v)$. Since the $U_0$-part of the
    component has changed, we have $h\geq2$. Hence, for some $r\in[h]$,
    \[
        2|D_r\cap U_0|\leq |C_j(v)\cap U_0|.
    \]
    Choose $u\in D_r\cap U_0$. Then
    \[
        C_{j-1}(u)=D_r
        \qquad\text{and}\qquad
        C_j(u)=C_j(v),
    \]
    so
    \[
        2|C_{j-1}(u)\cap U_0|
        \leq |C_j(u)\cap U_0|.
    \]
    Moreover, $C_i(u)=C_i(v)$ for every $i\geq j$, so $u$ inherits all the
    required inequalities from $v$ at the later stages. Thus $u$ witnesses
    the assertion for $j-1$, completing the induction step.
\end{proof}
\section{The thinning–sprinkling technique: a primer}\label{sec: connectivity}
This section provides a detailed illustration of how to use the thinning--sprinkling technique in a simple setting: the robustness of vertex connectivity under random vertex sampling. We prove Theorem~\ref{thm: subsampling k-connected}, which states that the subgraph induced by a $p$-random set of vertices in a $k$-connected graph is, with high probability, $\Omega(p^2k)$-connected. The result itself is not new and was proved in~\cite{connected_subsampling}, but we give a different, simpler proof and present it in detail, as it will serve as a model for the thinning and sprinkling arguments used throughout the paper.

The first lemma is the thinning step. Using the notation introduced in Section~\ref{sec: notation}, it says that if $G[V_p]$ does not have the desired connectivity with high enough probability, then, with still non-negligible probability, $G[V_{p/2}]$ has two distinct components, each meeting $V_{p/4}$.

\begin{lemma}\label{lem: k-connected thinning}
Let $0<\varepsilon\leq 1/8$, let $G$ be a $k$-connected graph and suppose that $\varepsilon p^2k\geq 4$. If 
\[
    \mathbb P\bigl(G[V_p]\text{ is not $\varepsilon p^2k$-connected}\bigr)
    \geq2^{-\varepsilon p^2k},
\]
then, with probability at least $2^{-4\varepsilon p^2k}$, $G[V_{p/2}]$ has two distinct connected components, each meeting $V_{p/4}$.
\end{lemma}

\begin{proof}
Let $\mathcal E_1$ be the event that $V_p$ contains at least $\varepsilon p^2k+1$ vertices but $G[V_p]$ is not $\varepsilon p^2k$-connected and let $\mathcal E_2$ be the event that $G[V_{p/2}]$ has two distinct connected components, each meeting $V_{p/4}$. Condition on an outcome of $V_p$ for which $\mathcal E_1$ occurs, and let $S\subseteq V_p$ be a cut-set witnessing this. Choose $u,v\in V_p\setminus S$ lying in distinct components of $G[V_p]-S$. Under the thinning coupling (described in the beginning of Section 2), we have $S\cap V_{p/2}=\emptyset$ and $u,v\in V_{p/4}$ with probability $2^{-|S|-4}\geq 2^{-2\varepsilon p^2k}$, where we used $\varepsilon p^2k\geq 4$. Whenever this happens, $u$ and $v$ lie in distinct components of $G[V_{p/2}]$, and both components meet $V_{p/4}$. Thus, $\mathbb P\bigl(\mathcal E_2\mid\mathcal E_1\bigr)\geq 2^{-2\varepsilon p^2k}$. By a Chernoff bound, we have that $V_p$ contains at least $pk/2\geq\varepsilon p^2k+1$ vertices with probability at least $1-e^{-pk/8}\ge 1-2^{-1-\varepsilon p^2 k}$, using $\varepsilon\leq 1/8$ for the last inequality. By assumption, $G[V_p]$ is not $\varepsilon p^2k$-connected with probability at least $2^{-\varepsilon p^2k}$, and so inclusion-exclusion gives $\mathbb P\bigl (\mathcal E_1\bigr )\geq 2^{-2\varepsilon p^2k}$. Consequently, $\mathbb P\bigl (\mathcal E_2\bigr )\geq \mathbb P\bigl (\mathcal E_1\cap\mathcal E_2\bigr )=\mathbb P\bigl (\mathcal E_1\bigr )\cdot\mathbb P\bigl (\mathcal E_2\mid\mathcal E_1\bigr )\geq 2^{-4\varepsilon p^2k}$.
\end{proof}

Next we use the sprinkling step to show that the obstruction produced by the thinning lemma is in fact too unlikely to occur. Together with Lemma~\ref{lem: k-connected thinning}, this will immediately imply Theorem~\ref{thm: subsampling k-connected}.

\begin{lemma}\label{lem: k-connected sprinkling}
There exists $0<\varepsilon\leq 1/8$ such that, whenever $p^2k=\omega(\log n)$, the probability that $G[V_{p/2}]$ has two distinct connected components, each meeting $V_{p/4}$, is less than $2^{-4\varepsilon p^2k}$.
\end{lemma}

\begin{proof}
Set $m= p^2k/64$ and $q=8/(pk)$. Let $U_0=V_{p/4}$, and let $U_1,\ldots,U_m$ be $q$-random subsets of $V(G)$, independent of each other and $U_0$. Recall from Section~\ref{sec: notation} that we write $W_i=U_0\cup\ldots\cup U_i$, $G_i=G[W_i]$, $C_i(v)$ denotes the component of $v$ in $G_i$ and a component of $G_i$ is seeded if it meets $U_0$.

\begin{claim}\label{clm: connected sprinkling claim 1}
    There exists an absolute constant $c_1>0$ such that, with probability at least $1-e^{-c_1p^2k}$, every two distinct seeded components of $G_m$ are joined in $G$ by at least $k/2$ internally vertex-disjoint paths of length at most~$3$.
\end{claim}
\begin{proof}
    Fix $u,v\in V(G)$. For $0\leq i\leq m-1$, call round $i+1$ successful if one of the following holds: if $u$ or $v$ is not in $U_0$, if $C_{i}(u)=C_{i}(v)$, if there already are $k/2$ internally vertex-disjoint paths of length at most $3$ between $C_{i}(u)$ and $C_{i}(v)$, or if $C_{i}(u)\cap U_0\neq C_{i+1}(u)\cap U_0$. We claim that, conditioned on the exploration $\mathcal{F}_{i}(u,v)$, round $i+1$ is successful with probability at least $1/2$.

Let us reveal the information associated with $\mathcal{F}_{i}(u,v)$ and suppose that the first three possibilities do not hold. Since $G$ is $k$-connected, Menger's theorem gives $k$ internally vertex-disjoint paths between $C_i(u)$ and $C_i(v)$. Choose such a family with minimum total length. Since fewer than $k/2$ of these paths have length at most $3$, at least $k/2$ have length at least $4$. For each such path, write its first three vertices, starting from $C_i(u)$, as $axy$. Note that $x\in N_G(C_i(u))$ while, by minimality, $y$ lies outside of both $C_i(u)\cup N_G(C_i(u))$ and $C_i(v)\cup N_G(C_i(v))$. Thus, the event $x\in U_{i+1}$ has probability $q$, while the event $y\in U_0$ has probability $p/4$ and is not revealed by $\mathcal F_i(u,v)$. Hence they occur simultaneously with probability $q\cdot p/4=2/k$. The relevant vertices are distinct for the different paths, so these events are independent. Whenever one such event occurs, $y$ joins the component of $u$ in $G_{i+1}$ and therefore $C_i(u)\cap U_0\neq C_{i+1}(u)\cap U_0$. Consequently, the probability that round $i+1$ is successful is at least $1-(1-2/k)^{k/2}\geq1-e^{-1}>1/2$.

Let $X_{u,v}$ be the number of successful rounds. Observe that the event of success in round $i+1$ is $\mathcal F_{i+1}(u,v)$-measurable and we already showed $\mathbb P\bigl(\text{Success in round $i+1$}\mid\mathcal F_{i}(u,v)\bigr )\geq 1/2$. As laid out in Section~\ref{sec: notation}, $X_{u,v}$ then stochastically dominates $\operatorname{Bin}(m,1/2)$ with $m=p^2k/64$. Recalling that $p^2k=\omega(\log n)$, a standard binomial lower-tail estimate gives $\mathbb P(X_{u,v}\leq\log_2 n)\leq e^{-p^2k/128}\leq n^{-3}$. A union bound over all pairs $u,v$ therefore shows that, for some absolute constant $c_1>0$, with probability at least $1-e^{-c_1p^2k}$, we have $X_{u,v}>\log_2 n$ for every $u,v\in V(G)$.

Assume this event holds, and suppose that $A,B$ are two distinct seeded components of $G_m$ which are joined by fewer than $k/2$ internally vertex-disjoint paths of length at most $3$. By Lemma~\ref{lem: minimal vertex}, applied to $G_m$, there exists $u\in A\cap U_0$ such that, for every $i\leq m-1$, either $C_i(u)\cap U_0=C_{i+1}(u)\cap U_0$ or $2|C_i(u)\cap U_0|\leq|C_{i+1}(u)\cap U_0|$. Let $v$ be any vertex in $B\cap U_0$. Since $C_i(u)\subseteq A$ and $C_i(v)\subseteq B$ for every $i\leq m$, these two components never coincide. Moreover, if at some earlier stage there were $k/2$ internally vertex-disjoint paths of length at most $3$ between them, the same paths would connect $A$ and $B$, contrary to our assumption. Hence every successful round for the pair $u,v$ satisfies $C_i(u)\cap U_0\neq C_{i+1}(u)\cap U_0$, implying that $2|C_i(u)\cap U_0|\leq|C_{i+1}(u)\cap U_0|$. There are more than $\log_2 n$ such rounds which is impossible. This proves Claim~\ref{clm: connected sprinkling claim 1}.
\end{proof}

\begin{claim}\label{clm: connected sprinkling claim 2}
    Condition on an outcome of $U_0,\ldots,U_m$ satisfying Claim~\ref{clm: connected sprinkling claim 1}, and let $U_{m+1}$ be an independent $p/8$-random subset of $V(G)$. There exists an absolute constant $c_2>0$ such that, with probability at least $1-e^{-c_2p^2k}$, all seeded components of $G_m$ lie in one component of $G_{m+1}$.
\end{claim}
\begin{proof}
    Fix two distinct seeded components $A,B$ of $G_m$. By Claim~\ref{clm: connected sprinkling claim 1}, there are $k/2$ internally vertex-disjoint paths of length at most $3$ between them. Each such path has at most two internal vertices. Every internal vertex belongs to $U_{m+1}$ independently with probability $p/8$. Thus, each path is contained in $G_{m+1}$ with probability at least $p^2/64$. Since the paths are internally vertex-disjoint, these events are independent, and therefore the probability that none of them joins $A$ to $B$ is at most $(1-p^2/64)^{k/2}\leq e^{-c_2p^2k}$ for some absolute constant $c_2>0$. There are at most $n^2$ pairs of seeded components. Since $p^2k=\omega(\log n)$, a union bound proves Claim~\ref{clm: connected sprinkling claim 2}.
\end{proof}
It remains only to couple these sprinklings with $V_{p/2}$. Recall the definition of $p'$ and $V_{p'}$ in Section~\ref{sec: notation} and note that $p'\geq p/4$. Put $R=U_1\cup\cdots\cup U_{m+1}$. The set $R$ is $\rho$-random for some $\rho$ satisfying $\rho\leq mq+p/8= p/4$. Hence we may couple $R$ as a subset of $V_{p'}$. On the event supplied by Claims~\ref{clm: connected sprinkling claim 1} and~\ref{clm: connected sprinkling claim 2}, all components meeting $U_0=V_{p/4}$ already lie in one component of $G_{m+1}$, and therefore also in one component of $G[V_{p/2}]$. Thus, choosing $\varepsilon>0$ sufficiently small, the probability that $G[V_{p/2}]$ has two distinct components, each meeting $V_{p/4}$, is less than $2^{-4\varepsilon p^2k}$.
\end{proof}

\section{Subsampling tough graphs} 
    
    In this section we adapt the thinning--sprinkling technique to tough graphs and use it to prove Theorem~\ref{thm: subsample tough}, which says that a $p$-random induced subgraph of a $t$-tough graph is, with very high probability, $\Omega(p^2t)$-tough, provided $p^2t$ is large compared to $\log n$. Throughout, we rely on the notation introduced in Section~\ref{sec: notation}. We begin with the thinning argument, which is essentially the same as for connectivity, except that a failure of toughness may give rise to an arbitrary number of seeded components. 
    \begin{lemma}\label{lem: tough thinning}
    Let $\varepsilon>0$, let $G$ be a $t$-tough graph and suppose that $\varepsilon p^2t\geq 2$. If
    \[
    \mathbb P\bigl(G[V_p]\text{ is not $\varepsilon p^2t$-tough}\bigr)\geq 2^{-\varepsilon p^2t},
    \]
    then there exists $k\geq 2$ such that, with probability at least $2^{-4k\cdot \varepsilon p^2 t}$, $G[V_{p/2}]$ has $k$ distinct connected components, each meeting $V_{p/4}$.
\end{lemma}
\begin{proof}
     Let $\mathcal E_1$ be the event that $G[V_p]$ is not $\varepsilon p^2t$-tough and, for $2\leq k\leq n$, let $\mathcal E_{1,k}$ be the event that there exists $S\subseteq V_p$ of size less than $k\cdot \varepsilon p^2t$ such that $G[V_p]-S$ has at least $k$ distinct connected components. Note that if $\mathcal E_1$ happens then so does $\mathcal E_{1,k}$ for some $k\geq 2$. Therefore, we have $\mathbb P\bigl(\mathcal E_1\bigr)\leq\sum_{k=2}^n\mathbb P\bigl(\mathcal E_{1,k}\bigr)$. Since $\sum_{k=2}^n2^{-k}<1$, we can fix $2\leq k\leq n$ for which $\mathbb P\bigl(\mathcal E_1\bigr)/2^{k}\leq \mathbb P\bigl(\mathcal E_{1,k}\bigr)$. Let $\mathcal E_{2}$ be the event that $G[V_{p/2}]$ contains at least $k$ distinct connected components, each meeting $V_{p/4}$. 
     
    Condition on $V_p$ for which $\mathcal E_{1,k}$ occurs and let $S\subseteq V_p$ be a cut-set witnessing this. Let $v_1,\ldots,v_k\in V_p\setminus S$ be vertices from distinct components of $G[V_p]-S$. Under the thinning coupling, we have $S\cap V_{p/2}=\emptyset$ and $v_1,\ldots,v_k\in V_{p/4}$ with probability $2^{-|S|-2k}\geq 2^{-2k\cdot \varepsilon p^2 t}$, where we used $\varepsilon p^2t\geq 2$. Whenever this happens, $v_1,\ldots,v_k$ lie in distinct components of $G[V_{p/2}]$, each meeting $V_{p/4}$. Thus, $\mathbb{P}\bigl(\mathcal E_{2}\mid\mathcal E_{1,k}\bigr)\geq 2^{-2k\cdot \varepsilon p^2 t}$. We also have $\mathbb P\bigl(\mathcal E_{1,k}\bigr)\geq \mathbb P\bigl(\mathcal E_1\bigr)/2^{k}\geq 2^{-2k\cdot \varepsilon p^2t}$, again using $\varepsilon p^2t\geq 2$. Consequently,
    $\mathbb{P}\bigl(\mathcal E_{2}\bigr)\geq \mathbb P\bigl(\mathcal E_{1,k}\cap \mathcal E_2\bigr)=\mathbb P\bigl(\mathcal E_{1,k}\bigr)\cdot \mathbb P\bigl(\mathcal E_2\mid\mathcal E_{1,k}\bigr)\geq 2^{-4k\cdot \varepsilon p^2t}.
    $
 \end{proof}
    Next, we turn to the sprinkling part of the argument, which, together with Lemma~\ref{lem: tough thinning}, immediately implies Theorem~\ref{thm: subsample tough}.
\begin{lemma}\label{lem: toughness sprinkling}
    There exists $\varepsilon>0$ such that whenever $p^2t=\omega(\log n)$ the following holds for every $k\geq 2$. The probability that $G[V_{p/2}]$ contains $k$ distinct connected components, each meeting $V_{p/4}$, is less than $2^{-4k\cdot \varepsilon p^2t}$.
\end{lemma}
The sprinkling step in the previous section (Lemma~\ref{lem: k-connected sprinkling}) used Menger's theorem. For toughness, we no longer have such a tool. Instead we use the following elementary lemma.
    \begin{lemma}\label{lem: menger for tough}
        Let $G$ be a $t$-tough graph and let $V_1,\ldots,V_k$ be disjoint nonempty subsets of $V(G)$ with no edges between distinct $V_i$, where $k\geq2$. Then, there exist paths $P_1,\ldots, P_{kt/2}$ only intersecting on $\bigcup V_j$ where every $P_i$ either is a path of length at most $3$ joining two distinct sets from $V_1,\ldots,V_k$, or $P_i$ is of length $2$ with one endpoint in $V_j$, for some $j$, and the other endpoint is outside of $\bigcup_{h=1}^k\bigl(V_h\cup N_G(V_h)\bigr)$.
    \end{lemma}
    \begin{proof}
        Let $P_1,\ldots,P_\ell$ be a maximal collection of paths of one of the two types described above, intersecting only on $\bigcup V_j$. Suppose that $\ell<kt/2$ as otherwise we are done. Let $W$ be the set of vertices appearing in $P_1,\ldots,P_\ell$ but not in $V_1,\ldots,V_k$. Then, $|W|\leq 2\ell<kt$. Since $G$ is $t$-tough, $G-W$ has fewer than $k$ connected components. Hence, there exists a path in $G$ with one endpoint in $V_i$ and the other in $V_j$ for some $i\neq j$. Let $P$ be a shortest path in $G-W$ joining two distinct sets among $V_1,\ldots,V_k$. Note that $P$ has length at least $4$, as otherwise we get a contradiction to the maximality of $\ell$. Let $v_1,v_2,v_3$ be an initial segment of $P$ where $v_1\in V_j$ for some $j$. Since $P$ was chosen as short as possible, we have that $v_3\notin \bigcup_{h=1}^k\bigl(V_h\cup N_G(V_h)\bigr)$. Thus, $v_1v_2v_3$ is a path of the second type, again contradicting the maximality of $\ell$.
    \end{proof}
\begin{proof}[Proof of Lemma~\ref{lem: toughness sprinkling}]
    Let us fix some $k\geq 2$. Set $m=k\cdot p^2t/128$ and $q=16/(k\cdot pt)$. Let $U_0=V_{p/4}$, and let $U_1,\ldots,U_m$ be $q$-random subsets of $V(G)$, independent of each other and $U_0$.
    \begin{claim}\label{clm: tough sprinkling claim 1}
        There exists an absolute constant $c_1>0$ such that, with probability at least $1-e^{-k\cdot c_1p^2 t}$, every collection of $k$ distinct seeded components of $G_m$ is joined in $G$ by at least $kt/4$ internally vertex-disjoint paths of length at most $3$, each joining two distinct components of the collection.
    \end{claim}
    \begin{proof}
        Fix $v_1,\ldots,v_k\in V(G)$. For $0\leq i\leq m-1$, call round $i+1$ successful if one of the following holds: if at least one of $v_1,\ldots,v_k$ is not in $U_0$, if $C_i(v_j)=C_i(v_h)$ for some $j\neq h$, if there are $tk/4$ internally vertex-disjoint paths of length at most $3$ each joining two distinct components from $C_i(v_1),\ldots,C_i(v_k)$, or if $C_i(v_j)\cap U_0\neq C_{i+1}(v_{j})\cap U_0$ for some $j$. First, we claim that, conditioned on exploration history $\mathcal F_i(v_1,\ldots, v_k)$, round $i+1$ is successful with probability at least $1/2$.

        Let us reveal the information associated with $\mathcal F_i(v_1,\ldots, v_k)$ and suppose that the first three possibilities do not hold. Let $P_1,\ldots, P_{kt/2}$ be a collection of paths as guaranteed by Lemma~\ref{lem: menger for tough} applied with $C_i(v_1),\ldots,C_i(v_k)$. Since we assumed that there is no collection of $kt/4$ internally vertex-disjoint paths joining two distinct components from $C_i(v_1),\ldots,C_i(v_k)$, it follows that at least $kt/4$ of them are of the form $axy$ with $a\in C_i(v_j)$, for some $j$, and $y\notin\bigcup_{h=1}^k\bigl(C_i(v_h)\cup N_G(C_i(v_h))\bigr)$. Thus, the event $x\in U_{i+1}$ happens with probability $q$ and $y\in U_0$ with probability $p/4$ since both events are independent of the information revealed by $\mathcal F_i(v_1,\ldots,v_k)$. Hence, they occur simultaneously with probability $q\cdot p/4=4/(kt)$. As the paths are internally vertex-disjoint, these events are independent for every path. Furthermore, whenever one such event occurs then $C_{i}(v_j)\cap U_0\neq C_{i+1}(v_j)\cap U_0$ since only the latter contains $y$. Consequently, round $i+1$ is successful with probability at least $1-(1-4/(kt))^{kt/4}\geq 1-e^{-1}>1/2$.

        Let $X_{v_1,\ldots,v_k}$ be the number of successful rounds. Observe that the event of success in round $i+1$ is $\mathcal F_{i+1}(v_1,\ldots,v_k)$-measurable and we already showed $\mathbb P\bigl(\text{Success in round $i+1$}\mid\mathcal F_i(v_1,\ldots,v_k)\bigr)\geq 1/2$. It follows that $X_{v_1,\ldots,v_k}$ stochastically dominates $\operatorname{Bin}(m,1/2)$. Recalling that $m=k\cdot p^2t/128$ and $p^2t=\omega(\log n)$, a standard binomial lower-tail estimate gives $\mathbb P\bigl(X_{v_1,\ldots,v_k}\leq k\log_2 n\bigr)\leq e^{-c_0k p^2t}=n^{-\omega(k)}$ for some absolute constant $c_0>0$. A union bound over all choices of $v_1,\ldots,v_k$ therefore shows that, for some absolute constant $c_1>0$, with probability at least $1-e^{-k\cdot c_1p^2t}$, we have $X_{v_1,\ldots,v_k}>k\log_2 n$ for all $v_1,\ldots,v_k\in V(G)$.

        Assume this event holds and suppose that $C_1,\ldots,C_k$ are distinct seeded components of $G_m$ which are joined by fewer than $kt/4$ internally vertex-disjoint paths of length at most $3$, each joining two distinct components among $C_1,\ldots,C_k$. By Lemma~\ref{lem: minimal vertex} applied to $G_m$ there exists, for every $j$, $v_j\in C_j\cap U_0$ such that for every $i\leq m-1$, either $C_i(v_j)\cap U_0=C_{i+1}(v_j)\cap U_0$ or $2|C_i(v_j)\cap U_0|\leq |C_{i+1}(v_j)\cap U_0|$. By our choice, all of $v_1,\ldots, v_k$ are in $U_0$. Also, since $C_i(v_j)\subseteq C_j$, these components are distinct for every $i$. Moreover, if at some earlier stage there were $kt/4$ internally vertex-disjoint paths of length at most $3$, each joining two distinct components among $C_i(v_1),\ldots,C_i(v_k)$, the same paths would join two distinct components among $C_1,\ldots,C_k$, contrary to our assumption. Hence, every successful round for $v_1,\ldots, v_k$ satisfies $C_i(v_j)\cap U_0\neq C_{i+1}(v_j)\cap U_0$ for some $j$. By the choice of $v_j$, $2|C_i(v_j)\cap U_0|\leq |C_{i+1}(v_j)\cap U_0|$ holds in such rounds. There are more than $k\log_2 n$ successful rounds, so for some specific $j$ there are more than $\log_2 n$ indices $i$ for which $2|C_i(v_j)\cap U_0|\leq |C_{i+1}(v_j)\cap U_0|$. It follows that $|C_j|>n$, a contradiction. This proves Claim~\ref{clm: tough sprinkling claim 1}.
    \end{proof}
    \begin{claim}\label{clm: tough sprinkling claim 2}
        Condition on an outcome of $U_0,\ldots, U_m$ satisfying Claim~\ref{clm: tough sprinkling claim 1}, and let $U_{m+1}$ be an independent $p/8$-random subset of $V(G)$. There exists an absolute constant $c_2>0$ such that, with probability at least $1-e^{-k\cdot c_2p^2t}$, $G_{m+1}$ has at most $k-1$ seeded components.
    \end{claim}
    \begin{proof}
        Fix $k$ distinct seeded components $C_1,\ldots,C_k$ of $G_m$. By Claim~\ref{clm: tough sprinkling claim 1}, there are $kt/4$ internally vertex-disjoint paths of length at most $3$, each joining two distinct components among $C_1,\ldots,C_k$. Every internal vertex belongs to $U_{m+1}$ independently with probability $p/8$. Thus, each path is contained in $G_{m+1}$ with probability at least $p^2/64$. Since the paths are internally vertex-disjoint, these events are independent, and therefore the probability that none of them joins two components among $C_1,\ldots,C_k$ in $G_{m+1}$ is at most $(1-p^2/64)^{kt/4}\leq e^{-c_0kp^2t}$ for some absolute constant $c_0>0$. There are at most $n^k$ choices of $k$ distinct seeded components of $G_m$. Since $p^2 t=\omega(\log n)$, a union bound shows that, for some absolute constant $c_2>0$, with probability at least $1-e^{-k\cdot c_2p^2t}$, every such collection has two components which merge in $G_{m+1}$. If $G_{m+1}$ had $k$ distinct seeded components, choosing one seeded component of $G_m$ inside each of them would contradict this property. This proves Claim~\ref{clm: tough sprinkling claim 2}.
    \end{proof}
    It only remains to couple these sprinklings with $V_{p/2}$. Put $R=U_1\cup \ldots\cup U_{m+1}$. The set $R$ is $\rho$-random for some $\rho$ satisfying $\rho\leq mq+p/8=p/4$. Using that $p'\geq p/4$, we may therefore couple $R$ as a subset of $V_{p'}$. Combining Claims~\ref{clm: tough sprinkling claim 1} and~\ref{clm: tough sprinkling claim 2}, for some absolute constant $c>0$, with probability at least $1-e^{-ckp^2t}$, $G_{m+1}$ has at most $k-1$ components meeting $U_0=V_{p/4}$. Therefore, $G[V_{p/2}]$ also has at most $k-1$ components meeting $V_{p/4}$. Thus, choosing $\varepsilon>0$ sufficiently small, the probability that $G[V_{p/2}]$ has $k$ distinct components, each meeting $V_{p/4}$, is less than $2^{-4k\cdot \varepsilon p^2t}$.
\end{proof}

\section{Subsampling vertex-transitive graphs}
Here, we adapt the thinning--sprinkling technique to $d$-regular vertex-transitive graphs and prove Theorem~\ref{thm: vertex-transitive subsampling}. The proof needs several new ingredients. We show that vertex-transitive graphs admit an automorphism-invariant decomposition into pieces which retain almost all of the degree and have useful expansion properties. This will be used to prove that $d$-regular vertex-transitive graphs satisfy a certain form of $d$-toughness and, in turn, have many short connecting paths. We then feed this deterministic structure into the thinning--sprinkling argument.

\subsection{Weak regularity for vertex-transitive graphs}

Regularity lemmas typically assert that a large graph can be decomposed into pieces exhibiting a substantial degree of uniformity. The classical example is Szemer\'edi's regularity lemma \cite{szemeredi-regularity}, which in the dense setting produces a partition into a bounded number of parts for which most pairs of parts are quasirandom. In the sparse setting, and especially when one asks the decomposition to respect additional structure or symmetry, such strong quasirandomness cannot in general be expected. Nevertheless, weaker forms of regularity may still survive: one can sometimes decompose the graph into structured pieces which retain most of the original degree and satisfy useful expansion or connectivity properties. For a classical weak regularity paradigm, see Frieze and Kannan \cite{frieze-kannan}.

This philosophy has appeared in several forms for regular and vertex-transitive graphs. In the dense vertex-transitive setting, Christofides, Hladk\'y and M\'ath\'e obtained an automorphism-invariant decomposition into isomorphic vertex-transitive pieces with strong robust connectivity properties \cite{christofides-hladky-mathe}. Related robust-component decompositions for dense regular graphs were developed by K\"uhn, Lo, Osthus and Staden \cite{kuhn-lo-osthus-staden}. More recently, Bedert, Buci\'c, Kravitz, Montgomery and M\"uyesser proved a weak nonabelian arithmetic regularity lemma for Cayley graphs, decomposing them into isomorphic mildly quasirandom pieces while retaining almost all of the generating set \cite{bedert-bucic-kravitz-montgomery-muyesser}. This type of decomposition was subsequently used by Bedert, Dragani\'c, M\"uyesser and Pavez-Sign\'e in their work on Hamilton cycles in moderately dense Cayley graphs \cite{bedert-draganic-muyesser-pavez-signe}.

The next lemma gives a particularly simple degree-scale version of this principle for arbitrary vertex-transitive graphs, with no density assumption. Given a vertex-transitive graph $G$, we say that
$B_1\sqcup\ldots\sqcup B_\ell=V(G)$ is a partition into \emph{blocks}
if every automorphism of $G$ permutes $B_1,\ldots,B_\ell$. Notice that each $G[B_i]$ is itself vertex-transitive: if $u,v\in B_i$, then an automorphism sending $u$ to $v$ must map $B_i$ to the block containing $v$, and hence to $B_i$.

\begin{lemma}\label{lem: weak regularity}
    Let $0<\varepsilon\leq1/4$ and let $G$ be a $d$-regular
    vertex-transitive graph. Then there exists a partition
    $B_1,\ldots,B_\ell$ into blocks such that every $G[B_i]$ is regular
    of degree at least $(1-\varepsilon)d$, and 
    $e(U,B_i\setminus U)
        \geq \min\{d|U|/4,\varepsilon d^2/10\}$
    for every
    $U\subseteq B_i$ with $|U|\leq |B_i|/2$.
\end{lemma}

To prove the lemma, we first show how to extract from an arbitrary graph a dense expanding subgraph.

\begin{definition}
    A graph $H$ is a $(t,s)$-expander if for every $U\subseteq V(H)$ with
    $|U|\leq |V(H)|/2$ it holds that
    $e_H(U,V(H) \setminus U)\geq \min\{t|U|,ts\}$.
\end{definition}

\noindent
The parameter $t$ is the expansion rate and $s$ is the scale up to which this rate is required. Sets of size at most $s$ have edge-boundary at least $t|U|$, while for larger sets we only require the boundary to be of size at least $ts$. Thus, unlike the usual notion of an expander, we require linear expansion only for relatively small sets, while retaining a nontrivial lower bound on the boundary of all larger sets.

\begin{proposition}\label{prop: sublinear expander}
    Let $G$ be a graph with average degree $d>0$. For every
    $0<\varepsilon\leq1/4$, $G$ contains a
    $(d/10,\varepsilon d)$-expander $H$ with average degree at least
    $(1-\varepsilon)d$ and minimum degree greater than $d/2$.
\end{proposition}

\begin{proof}
Choose $V\subseteq V(G)$ of minimum size subject to $|V|\geq2d/5$ and
\[
    2e(G[V])\geq d|V|-\frac{2\varepsilon d^2}{5}.
\]
Such a set exists since $V(G)$ satisfies both conditions. Set $H=G[V]$. Then $d(H)\geq(1-\varepsilon)d\geq3d/4$. Fix $v\in V$. Minimality gives
\[
    2e(H)-2d_H(v)<(|V|-1)d-\frac{2\varepsilon d^2}{5}.
\]
Comparing this with the inequality defining $H$ yields $d_H(v)>d/2$.

Now let $\varnothing\neq U\subseteq V$ with $|U|\leq |V|/2$, and put
$W=V\setminus U$. If $|U|<2d/5$, the minimum-degree bound gives
$e_H(U,W)>(d/2-|U|)|U|>d|U|/10$.
Otherwise,  $2d/5 \leq |U| \leq |W|$, so minimality implies
\[
    2e(H[U])+2e(H[W])
    <\left(d|U|-\frac{2\varepsilon d^2}{5}\right)
      +\left(d|W|-\frac{2\varepsilon d^2}{5}\right)\\
    =d|V|-\frac{4\varepsilon d^2}{5}.
\]
Since $2e(H)\geq d|V|-2\varepsilon d^2/5$, this gives
$e_H(U,W)>\varepsilon d^2/5$. In either case, the required expansion
bound follows.
\end{proof}
\begin{proof}[Proof of Lemma~\ref{lem: weak regularity}]
    Let $H\subseteq G$ be the $(d/10,\varepsilon d)$-expander given by
    Proposition~\ref{prop: sublinear expander}. Let $\Gamma$ be the
    spanning subgraph of $G$ whose edge set is the orbit of $E(H)$ under the automorphisms of $G$. Since $\Gamma$ is invariant under every
    automorphism of $G$, it is $d'$-regular for some $d'$. Moreover,
    $\Gamma$ contains $H$, so $d'\geq d(H)\geq(1-\varepsilon)d$. The
    connected components $B_1,\ldots,B_\ell$ of $\Gamma$ form a partition
    into blocks. Since $G[B_i]$ contains $\Gamma[B_i]$ and is vertex-transitive, it is regular of degree at least $d'$.

    Fix a block $B$ and $U\subseteq B$ with $|U|\leq |B|/2$. If
    $|U|\leq d/2$, then
    $e_\Gamma(U,B\setminus U)\geq (d'-|U|)|U|\geq d|U|/4$, as required.
    Suppose therefore that $|U|>d/2$.

    Let $X$ be the set of vertices of $B$ having at least $d/10$
    neighbors across the cut $(U,B\setminus U)$ in $\Gamma$. If
    $|X|\geq d/2$, then there are at least $d^2/40$ edges across the
    cut, which is enough. Hence, let us assume that
    $|X|<d/2$. Next, we use the classical fact that a connected $d'$-regular vertex-transitive graph is at least $2(d'+1)/3$-vertex-connected \cite{watkins-connectivity}. Since $d'\geq3d/4$ and $\Gamma[B]$ is connected, it is more than $d/2$-connected. Thus not every edge across the cut can
    be incident with $X$.

    Choose an edge $uv$ across the cut with $u,v\notin X$. Since
    $uv\in E(\Gamma)$, it lies in an automorphic copy $H'$ of $H$.
    The graph $H'$ is connected and hence lies inside $\Gamma[B]$.
    Both $u$ and $v$ have degree greater than $d/2$ in $H'$, while
    each has fewer than $d/10$ neighbors across the cut. It follows
    that each side of the cut contains more than $2d/5$ vertices of
    $H'$. Applying the expansion property of $H'$ to the smaller side
    gives at least $\varepsilon d^2/10$ edges across the cut.
\end{proof}

\subsection{Toughness of vertex-transitive graphs}

As mentioned in the introduction, every connected vertex-transitive graph is $1$-tough.
The weak regularity lemma above, however, allows us to show the following substantially
stronger form of toughness, which will be crucial for our purposes.

\begin{lemma}\label{lem: vertex-transitive toughness}
    Let $G$ be a connected $d$-regular vertex-transitive graph and let
    $S\subseteq V(G)$ be such that $G-S$ is disconnected. Denote the
    components of $G-S$ by $X_1,\ldots,X_h$. Then
    \[
        |S|\geq\frac{1}{125}\sum_{i=1}^h\min\{|X_i|,d\}.
    \]
\end{lemma}

\begin{proof}
    Apply the weak regularity lemma with $\varepsilon=1/4$, and let
    $B_1,\ldots,B_\ell$ be the resulting blocks. All blocks have the same
    size, say $b$, and every $G[B_j]$ has degree at least $3d/4$; in
    particular, $b\geq3d/4$.
    In each block $B$, consider the largest non-empty intersection $X_i\cap B$,
    if one exists, keep this set and discard the others. Let $D$ be the set of discarded
    vertices from all the $X_i$, put $S'=S\cup D$, and let $a$ be the number of original components retaining any vertices.

     By maximality, every discarded intersection $U=X_i\cap B$ has $|U|\leq b/2$, and
    all its edges to $B\setminus U$ go to $S\cap B$, since every vertex of
    $B\setminus(U\cup S)$ lies in a component of $G-S$ different from $X_i$.
    By weak regularity,
    $e(U,S\cap B)\geq\min\{d|U|/10,d^2/40\}$, which is at least both
    $(d/40)\cdot \min\{|U|,d\}$ and $d^2|U|/(20b)$, where we use that $|U|\leq b/2$. Summing over the
    discarded intersections counts at most $d|S|$ edges, so
    \[
        A:=\sum_{U\text{ discarded}}\min\{|U|,d\}\leq40|S|,
        \qquad |D|\leq\frac{20b}{d}|S|.
    \]

Choose a spanning tree in the graph whose vertices are the blocks,
    with two blocks adjacent whenever an edge of $G$ joins them. This
    graph is connected because $G$ is connected. Represent each tree
    edge by one such edge of $G$, obtaining a set $F$ of $\ell-1$ edges.

    For now, let $\varphi$ be an arbitrary automorphism of $G$. Then, $\varphi$ permutes the blocks, so $\varphi(F)$
    again represents a spanning tree on the blocks. Delete the tree
    edges whose representatives in $G$ meet $S'$. Each remaining representative
    has retained endpoints belonging to the same $X_i$, since distinct
    components of $G-S$ have no edges between them. As each block retains
    vertices from at most one $X_i$, blocks retaining vertices from
    different $X_i$ lie in different components of the remaining forest.
    There are $a$ such $X_i$, so the forest has at least $a$ components.
    Hence at least $a-1$ edges of $\varphi(F)$ meet $S'$.

    Now, choose $\varphi$ uniformly at random. By vertex transitivity,
    each endpoint of a fixed edge of $F$ is mapped uniformly over the
    $\ell b$ vertices of $G$. The probability that the image of this edge
    meets $S'$ is therefore at most $2|S'|/(\ell b)$. Summing over the
    $\ell-1$ edges of $F$, the expected number of edges of $\varphi(F)$
    meeting $S'$ is at most $2(\ell-1)|S'|/(\ell b)\leq2|S'|/b$.
    Since this number is always at least $a-1$, we obtain
    $a-1\leq2|S'|/b$.

    The original components retaining vertices contribute at most $ad$
    to $\sum_i\min\{|X_i|,d\}$, while all others contribute at most $A$,
    since each is a union of discarded intersections. Since $S$
    disconnects $G$, the vertex-connectivity bound gives $|S|\geq d/2$.
    Using $b\geq3d/4$ and the bounds above, we conclude
    \[
        \begin{aligned}
        \sum_{i=1}^h\min\{|X_i|,d\}
        &\leq A+ad
        \leq40|S|+d+\frac{2d}{b}(|S|+|D|)\\
        &\leq\left(80+2+\frac83+40\right)|S|<125|S|,
        \end{aligned}
    \]
    completing the proof.
\end{proof}

The following connecting lemma, needed for sprinkling, is a consequence
of the toughness property of vertex-transitive graphs.

\begin{lemma}\label{lem: vertex-transitive connectors}
    Let $G$ be a connected $d$-regular vertex-transitive graph, let
    $0<r\leq d$, and let $C_1,\ldots,C_k$ be pairwise disjoint connected
    vertex sets with no edges between distinct $C_i$, where $k\geq2$ and
    $|C_i|\geq r$ for every $i$. Then there exist two collections
    $\mathcal P_1,\mathcal P_2$ of paths which intersect only on
    $\bigcup_iC_i$ with the following properties:
    \begin{itemize}
        \item every path in $\mathcal P_1$ has length $2$ and joins two
        distinct sets among $C_1,\ldots,C_k$;
        \item every path in $\mathcal P_2$ either has length $3$ and joins
        two distinct sets among $C_1,\ldots,C_k$, or has length $2$ with
        one endpoint in some $C_i$ and the other endpoint outside
        $\bigcup_i(C_i\cup N_G(C_i))$.
    \end{itemize}
    Writing $N_j=|\mathcal P_j|$, we have
    $\frac{d}{r}N_1+N_2\geq\frac{dk}{300}$.
\end{lemma}

\begin{proof}
    Let $X$ be the set of vertices outside $\bigcup_iC_i$ which have
    neighbors in at least two distinct sets among $C_1,\ldots,C_k$.
    For every $x\in X$, choose a path of length $2$ through $x$ joining
    two such sets. These paths form $\mathcal P_1$, so $N_1=|X|$.

    In $G-X$, let $\mathcal P_2$ be a maximal collection of paths of the
    second type in the statement, intersecting only on $\bigcup_iC_i$.
    Let $Z$ be the set of vertices appearing on these paths but not in
    $\bigcup_iC_i$. Thus $|Z|=2N_2$. Put $S=X\cup Z$.

    For each $i$, let $B_i=N_G(C_i)\setminus C_i$, set
    $Y_i=B_i\setminus S$, and put $Q_i=C_i\cup Y_i$. We claim that
    $Q_i$ is a component of $G-S$. It is connected by definition.
    If an edge $xy$ with $x\in Y_i$ left $Q_i$ in $G-S$, then either
    $y$ has a neighbor in some $C_j$ with $j\neq i$, giving a path of
    length $3$ from $C_i$ to $C_j$, or $y$ lies outside
    $\bigcup_j(C_j\cup N_G(C_j))$, giving an escape path of length $2$.
    Either path could be added to $\mathcal P_2$, contradicting its
    maximality. Thus $Q_1,\ldots,Q_k$ are distinct components of $G-S$.

    Put $s_i=|C_i|$. If $s_i<d$, regularity gives
    $e(C_i,B_i)\geq s_i(d-s_i)$. Since every vertex of
    $B_i\setminus X$ has at most $s_i$ neighbors in $C_i$, we have
    \[
        |B_i\setminus X|
        \geq\frac{e(C_i,B_i)-e(C_i,X)}{s_i}
        \geq d-s_i-\frac{e(C_i,X)}{s_i}.
    \]
    As $|Q_i|=s_i+|B_i\setminus(X\cup Z)|$, it follows that
    $\min\{|Q_i|,d\}\geq d-e(C_i,X)/r-|B_i\cap Z|$.
    The same inequality is immediate when $s_i\geq d$. Summing over $i$,
    and using that a vertex of $Z$ belongs to at most one $B_i$, gives
    \[
        \sum_{i=1}^k\min\{|Q_i|,d\}
        \geq dk-\frac drN_1-2N_2.
    \]

    Lemma~\ref{lem: vertex-transitive toughness} applied to
    $S=X\cup Z$ now gives
    \[
        N_1+2N_2
        \geq\frac1{125}
        \left(dk-\frac drN_1-2N_2\right).
    \]
    Rearranging and using $r\leq d$, we obtain
    \[
        dk\leq 125(N_1+2N_2)+\frac drN_1+2N_2
        \leq300\left(\frac drN_1+N_2\right).
    \]
    Dividing by $300$ proves the lemma.
\end{proof}

\subsection{The thinning--sprinkling technique on vertex-transitive graphs}

We now combine Lemma~\ref{lem: vertex-transitive connectors} with the
thinning--sprinkling technique. Fix $\varepsilon>0$ and assume $\varepsilon\leq1/4$, as otherwise we can take $\varepsilon=1/4$. Throughout this subsection all constants are allowed to depend on $\varepsilon$. 

We first prove the thinning step, which says that if $G[V_p]$ fails to
be $(1-\varepsilon)$-tough with sufficiently large probability, then,
with still non-negligible probability, $G[V_{p/2}]$ contains many
distinct components which are already robustly seeded inside $V_{p/4}$.

\begin{lemma}\label{lem: vertex-transitive thinning}
For every $0<\varepsilon\leq1/4$ the following holds for all small enough $c>0$. Let $G$ be a connected $d$-regular vertex-transitive
graph and let $p\in[0,1]$ satisfy $p^2d=\omega(\log n)$. If
\[
    \mathbb P\bigl(G[V_p]\text{ is not $(1-\varepsilon)$-tough}\bigr)
    \geq2^{-cp^2d},
\]
then there exists $k\geq2$ such that, with probability at least
$2^{-2ckp^2d}$, $G[V_{p/2}]$ has $k$ distinct connected components
$C$, each meeting $V_{p/4}$, such that every component of
$G[C\cap V_{p/4}]$ has size at least $\varepsilon pd/100$.
\end{lemma}

\begin{proof}
Let $\mathcal E_1$ be the event that $G[V_p]$ is not
$(1-\varepsilon)$-tough and, for $2\leq k'\leq n$, let
$\mathcal E_{1,k'}$ be the event that there exists $S\subseteq V_p$
with $|S|<(1-\varepsilon)k'$ such that $G[V_p]-S$ has $k'$
connected components. Since $\mathcal E_1$ implies
$\mathcal E_{1,k'}$ for some $k'$, we may fix $k'$ such that $\mathbb P(\mathcal E_{1,k'})
    \geq\mathbb P(\mathcal E_1)/n$.

We first restrict to the typical event $\mathcal D$ that every vertex has
between $(1-\varepsilon/6)pd$ and $(1+\varepsilon/6)pd$ neighbors in
$V_p$. Since $\varepsilon$ is fixed, Chernoff's inequality and a union
bound give $\mathbb P(\overline{\mathcal D})
    \leq e^{-\Omega_\varepsilon(pd)}$.
As $pd\geq p^2d=\omega(\log n)$, by choosing $c=c(\varepsilon)$ sufficiently small, we may assume that
\[
    \mathbb P(\mathcal E_{1,k'}\cap\mathcal D)
    \geq\frac12\mathbb P(\mathcal E_{1,k'})\geq 2^{-cp^2d-1-\log_2 n}.
\]

Fix an outcome of $V_p$ for which $\mathcal E_{1,k'}$ and $\mathcal D$ occur, and let
$S$ witness $\mathcal E_{1,k'}$. Let
$C_1,\ldots,C_{k'}$ be distinct components of $G[V_p]-S$. Call $C_i$
\emph{good} if every vertex of $C_i$ has at most
$(1-\varepsilon/3)pd$ neighbors in $S$, and let $I$ be the set of good
indices. Every index outside $I$ contributes at least
$(1-\varepsilon/3)pd$ edges from $S$ to $\bigcup_i C_i$, while
\[
    e\left(S,\bigcup_iC_i\right)
    \leq(1+\varepsilon/6)pd|S|
    \leq(1-5\varepsilon/6)pdk'.
\]
It follows that $|I|\geq\varepsilon k'/3$.

If $\varepsilon k'/3<2$, then $k'=O_\varepsilon(1)$. Since $pd =\omega(\log n)$, we may assume in this case that
$|S|<(1-\varepsilon/3)pd$, and therefore every $C_i$ is good. Thus,
putting $k=\max\{\varepsilon k'/3,2\}$,
we can in all cases choose $k$ good components, which we denote by
$C_1,\ldots,C_k$. For every $i\leq k$ and $v\in C_i$, all neighbors of $v$ in $V_p$ lie
in $C_i\cup S$. Since $C_i$ is good and $\mathcal D$ holds,
\[
    d_G(v,C_i)
    \geq(1-\varepsilon/6)pd-(1-\varepsilon/3)pd
    =\varepsilon pd/6.
\]

Choose one vertex $v_i\in C_i$ for every $i\leq k$. We now perform the
nested thinning $V_{p/4}\subseteq V_{p/2}\subseteq V_p$, obtained by two independent halvings. First we require $S\cap V_{p/2}=\varnothing$.
This occurs with probability $2^{-|S|}$.

Condition on this event. All vertices of
$C_1\cup\cdots\cup C_k$ still belong to $V_{p/4}$ independently with
probability $1/4$. Since every vertex of $C_i$ has at least
$\varepsilon pd/6$ neighbors in $C_i$, another Chernoff bound and a
union bound show that, 
with conditional probability at least $1/2$,
\[
    d_G(v,C_i\cap V_{p/4})\geq \varepsilon pd/100
\]
for every $i\leq k$ and every $v\in C_i$.

Whenever these events occur, let us select one component in $G[C_i\cap V_{p/2}]$ for every $i\leq k$. These components satisfy the conclusion of the lemma. Indeed, such components exist since every $C_i$ meets $V_{p/4}$ and they break into components of size at least $\varepsilon pd/100$ since $G[C_i\cap V_{p/4}]$ has minimum degree $\varepsilon pd/100$. Further, they are distinct as  $S\cap V_{p/2}=\varnothing$ and the
$C_i$ are distinct components of $G[V_p]-S$. Thus, if $\mathcal E_2$
denotes the event in the conclusion of the lemma, then
\[
    \mathbb P\bigl(\mathcal E_2\mid \mathcal E_{1,k'}\cap\mathcal D\bigr)
    \geq2^{-|S|-1}
    \geq2^{-k'-1}.
\]
Therefore,
$$
    \mathbb P(\mathcal E_2)
    \geq\mathbb P\bigl(\mathcal E_{1,k'}\cap \mathcal D\bigr)\cdot \mathbb P\bigl(\mathcal E_2\mid\mathcal E_{1,k'}\cap \mathcal D\bigr)\geq 2^{-6k/\varepsilon-1}\cdot 2^{-cp^2d-1-\log_2 n},
$$
where we use that the definition of $k$ gives $k'\leq6k/\varepsilon$.
Since
$p^2d=\omega(\log n)$, we get
\[
    \mathbb P(\mathcal E_2)\geq2^{-2ckp^2d},
\]
as required.
\end{proof}

\noindent
We next show that the obstruction produced by thinning is too unlikely to
occur. Recall the notation introduced in Section~\ref{sec: notation}. 

\begin{lemma}\label{lem: vertex-transitive sprinkling}
For every $0<\varepsilon\leq1/4$ the following holds for sufficiently small $c>0$. Let $p^2d=\omega(\log n)$ and $k\geq 2$. The probability that $G[V_{p/2}]$ has $k$ distinct connected components $C$, each meeting $V_{p/4}$, such that every component of $G[C\cap V_{p/4}]$ has size at least $\varepsilon pd/100$ is less than $2^{-2ckp^2d}$.
\end{lemma}

\begin{proof}
Fix $k\geq2$. Set
\[
   r=\varepsilon pd/100,\qquad t=\varepsilon d/150,\qquad m=k\cdot p^2t/128
    \qquad\text{and}\qquad
    q=16/(k\cdot pt).
\]
Note that $q\leq1$ since $pd=\omega(1)$. Let $U_0=V_{p/4}$, and let $U_1,\ldots,U_m$ be $q$-random subsets of $V(G)$, independent of each other and $U_0$. We call a component $C$ \emph{specially seeded} if $C$ meets $U_0$ and every component of $G[C\cap U_0]$ has size at least $r$. Recall from Section~\ref{sec: notation} that we write
$W_i=U_0\cup\cdots\cup U_i$, $G_i=G[W_i]$, and $C_i(v)$ denotes
the component of $v$ in $G_i$.

\begin{claim}\label{clm: vertex-transitive sprinkling paths}
There exists $c_1=c_1(\varepsilon)>0$ such that, with probability at
least $1-e^{-c_1k p^2t}$, every collection of $k$ distinct specially
seeded components $C_1,\ldots,C_k$ of $G_m$ is joined in $G$ by a
family of paths, intersecting only on $\bigcup_jC_j$, such that, if
$N_1$ of these paths have length $2$ and $N_2$ have length $3$, then
\[
    (d/r)N_1+N_2\geq kt/4.
\]
\end{claim}

\begin{proof}
Fix $v_1,\ldots,v_k\in V(G)$. For $0\leq i\leq m-1$, call round
$i+1$ successful if one of the following holds: some $C_0(v_j)$ has
size less than $r$; two of $C_i(v_1),\ldots,C_i(v_k)$ coincide; there
is already a family of paths joining distinct components among
$C_i(v_1),\ldots,C_i(v_k)$, intersecting only on their union, such
that $(d/r)N_1+N_2\geq tk/4$, where $N_1$ paths have length $2$ and
$N_2$ have length $3$; or
$C_i(v_j)\cap U_0\neq C_{i+1}(v_j)\cap U_0$ for some $j$. We claim
that, conditioned on the exploration history
$\mathcal F_i(v_1,\ldots,v_k)$, every round is successful with
probability at least $1/2$.

Reveal $\mathcal F_i(v_1,\ldots,v_k)$ and suppose that the first three
possibilities do not hold. The current components
$C_i(v_1),\ldots,C_i(v_k)$ are disjoint connected sets,
each of size at least $r$. Apply
Lemma~\ref{lem: vertex-transitive connectors}.
The paths in $\mathcal P_1$, together with those paths in
$\mathcal P_2$ of length $3$ joining two distinct current components,
would form a family of the type appearing in the third possibility.
Since that possibility does not hold, their contribution to
$(d/r)N_1+N_2$ is less than $tk/4$. On the other hand,
Lemma~\ref{lem: vertex-transitive connectors} gives total contribution
at least $tk/2$. Therefore there are at least $tk/4$ remaining paths in
$\mathcal P_2$, namely those of length $2$ whose other endpoint lies
outside
\[
    \bigcup_{h=1}^k
    \bigl(C_i(v_h)\cup N_G(C_i(v_h))\bigr).
\]

Each such path has the form $axy$, where $a\in C_i(v_j)$ for some $j$.
The event $x\in U_{i+1}$ has probability $q$, while $y\in U_0$ has
conditional probability $p/4$, since $\mathcal F_i(v_1,\ldots,v_k)$
only reveals $U_0$-membership in the current components and their
neighborhoods. Hence both occur with probability
$q(p/4)=4/(kt)$. The paths are disjoint outside the current
components, so these events are independent. Whenever one occurs, the
$U_0$-part of one of the current components grows. Consequently the
probability that round $i+1$ is successful is at least
\[
    1-\left(1-4/(kt)\right)^{tk/4}
    \geq1-e^{-1}>\frac12.
\]

Let $X_{v_1,\ldots,v_k}$ be the number of successful rounds. As explained in Section~\ref{sec: notation}, $X_{v_1,\ldots,v_k}$ stochastically dominates
$\operatorname{Bin}(m,1/2)$. Since $m=k\cdot p^2t/128=O_\varepsilon(k\cdot p^2d)$ and $p^2d=\omega(\log n)$, a standard binomial lower-tail estimate,
followed by a union bound over all $v_1,\ldots,v_k$, shows that for some
$c_1=c_1(\varepsilon)>0$, with probability at least
$1-e^{-c_1kp^2d}$ we have $X_{v_1,\ldots,v_k}>k\log_2n$ for every
choice of $v_1,\ldots,v_k$.

Assume this event holds and let $C_1,\ldots,C_k$ be distinct specially
seeded components of $G_m$ for which the conclusion of the claim
fails. By Lemma~\ref{lem: minimal vertex}, for every $j$ we may choose
$v_j\in C_j\cap U_0$ such that whenever
$C_i(v_j)\cap U_0$ changes from one round to the next, its size at
least doubles. For these vertices the first two possibilities are excluded
by construction, and the third is also impossible: truncating an earlier
connecting family at the enlarged components would give one at round
$m$ without decreasing $(d/r)N_1+N_2$. Hence every successful round doubles
the $U_0$-part of at least
one of the $k$ components. More than $k\log_2n$ successful rounds would
force one component to double more than $\log_2n$ times, a
contradiction. This proves the claim.
\end{proof}

\begin{claim}\label{clm: vertex-transitive final sprinkling}
Condition on an outcome of $U_0,\ldots,U_m$ satisfying
Claim~\ref{clm: vertex-transitive sprinkling paths}, and let
$U_{m+1}$ be an independent $p/8$-random subset of $V(G)$. There
exists $c_2=c_2(\varepsilon)>0$ such that, with probability at least
$1-e^{-c_2kp^2d}$, $G_{m+1}$ has at most $k-1$ specially seeded
components.
\end{claim}

\begin{proof}
Fix $k$ distinct specially seeded components
$C_1,\ldots,C_k$ of $G_m$. By the previous claim, there is a family of
internally disjoint paths joining distinct components such that
$(d/r)N_1+N_2\geq kt/4$, where $N_1$ paths have length $2$ and $N_2$
have length $3$. A path of length $2$ is contained in $G_{m+1}$ if
its unique internal vertex lies in $U_{m+1}$, an event of probability
$p/8$. A path of length $3$ is contained in $G_{m+1}$ if both of its
internal vertices lie in $U_{m+1}$, an event of probability $p^2/64$.
Since the paths are internally disjoint, these events are independent.
Thus the probability that none of the paths joins two of the components
is at most
\[
    \left(1-\frac p8\right)^{N_1}
    \left(1-\frac{p^2}{64}\right)^{N_2}
    \leq
    \exp\left\{-\frac p8N_1-\frac{p^2}{64}N_2\right\}
    \leq
    \exp\left\{-\frac {\varepsilon p^2}{800}\left((d/r)N_1+N_2\right)\right\},
\]
where we use $d/r=100/(\varepsilon p)$.
Recalling that
$(d/r)N_1+N_2\geq kt/4$ and $t=\Omega_\varepsilon(d)$, the last expression is at most
$e^{-c_0kp^2d}$ for some $c_0=c_0(\varepsilon)>0$.
Every specially seeded component of $G_{m+1}$ contains one of $G_m$,
since the components of $G[U_0]$ are fixed. There are at most $n^k$
choices of the $k$ components. Since $p^2 d=\omega(\log n)$, a union bound
proves the claim after taking $c_2>0$ sufficiently small.
\end{proof}

It remains only to couple these sprinklings with $V_{p/2}$. Put
$R=U_1\cup\cdots\cup U_{m+1}$. Since $\varepsilon\leq1/4$,
\[
    \mathbb P(v\in R)
    \leq mq+p/8
    = p/4.
\]
Hence we may couple $R$ as a subset of $V_{p'}$ (as defined in Section~\ref{sec: notation}). On the event supplied by the two claims,
$G_{m+1}$ has at most $k-1$ specially seeded components, and therefore
so does $G[V_{p/2}]$. This proves the lemma by choosing $c$ sufficiently small compared to $c_1$ and $c_2$.
\end{proof}

Combining the thinning--sprinkling steps, i.e., Lemmas~\ref{lem: vertex-transitive thinning} and~\ref{lem: vertex-transitive sprinkling}, with \(c\) sufficiently small, proves Theorem~\ref{thm: vertex-transitive subsampling}.

\section{Subsampling nearly regular sublinear expanders}
In this section we prove that a nearly regular sublinear expander is likely to stay an expander after subsampling. As a corollary, we will derive that nearly regular sublinear expanders are likely to be $(1-o(1))$-tough after subsampling. We could also prove this directly with an argument similar to that of the previous section. However, we believe that preserving expansion is of independent interest. Moreover, establishing this stronger conclusion allows us to illustrate a somewhat different sprinkling argument.

In this section, for $V\subseteq V(G)$, we write $\varphi(V)=V\cup N_G(V)$. For a collection $\mathcal C$ of components, we write $V(\mathcal C)$ for the union of their vertex sets. Through the proofs we assume that $n$ tends to infinity. We begin with the thinning step. Here $k$ denotes the size of a set witnessing failure of expansion.
\begin{lemma}\label{lem: expander thinning}
 Let $0<\lambda\leq1$ and $0<\varepsilon\leq 1/160$, let $G$ be a $\lambda$-expander on $n$ vertices with all degrees between $d$ and $(1+o(1))d$, and suppose that $p^2\lambda d=\omega(\log n)$. If
    \[
    \mathbb P\bigl(\text{$G[V_p]$ is not an $\varepsilon p\lambda$-expander}\bigr)\geq 2^{-\varepsilon p^2\lambda d},
    \]
    then there is $2\leq k\leq 3pn/5$ such that the following hold with probability at least $2^{-80\varepsilon p\lambda k}$. Every vertex in $V(G)$ has at most $2pd$ neighbors in $V_{p/4}$ and $\delta(G[V_{p/4}])\geq pd/64$. Further, there is a collection $\mathcal C$ of components of $G[V_{p/4}]$ with $|V(\mathcal C)|\geq k/5$ with the following property. Let $\mathcal C'$ be the components of $G[V_{p/2}]$ which meet $V(\mathcal C)$. Then, $|V(\mathcal C')|\leq 9k/17$ and $V(\mathcal C')$ is contained in $\varphi(V(\mathcal C))$.
\end{lemma}
\begin{proof}
Let $\mathcal E_1$ be the event that $G[V_p]$ is not an
$\varepsilon p\lambda$-expander and let $\mathcal E_{1,k}$ be the
event that there exists $U\subseteq V_p$ with
$|U|=k\leq |V_p|/2$ such that
$|N_G(U)\cap V_p|\leq\varepsilon p\lambda k$. If $\mathcal E_1$
occurs, then $\mathcal E_{1,k}$ occurs for some $k$. Hence, by
averaging, we can fix $k$ for which
$\mathbb P(\mathcal E_{1,k})\geq\mathbb P(\mathcal E_1)/n$.
Let $\mathcal E_2$ be the event in the conclusion of the lemma.
Finally, let $\mathcal E_3$ be the event that
$|V_p|\leq6pn/5$ and every vertex of $G$ has at least $pd/2$ and
at most $2pd$ neighbors in $V_p$.
Condition on an outcome of $V_p$ for which
$\mathcal E_{1,k}$ and $\mathcal E_3$ hold. Let $U\subseteq V_p$
witness $\mathcal E_{1,k}$ and put
$S=N_G(U)\cap V_p$, so that
$|S|\leq\varepsilon p\lambda k$. Note that
$pd/4\leq k\leq3pn/5$. Indeed, every vertex of $U$ has at least
$pd/2$ neighbors in $U\cup S$, so
$pd/2\leq k+|S|\leq2k$, while $k\leq|V_p|/2\leq3pn/5$. 

First, we remove a small additional set $S'$,
so that the remaining graph $G[V_p]-S-S'$ has large minimum degree.
Let $W\subseteq V_p\setminus S$ be of maximum size such that $G[W]$
has minimum degree at least $pd/8$, and set
$S'=V_p\setminus(S\cup W)$. We show that only few vertices were
lost in passing to $W$, that is, $|S'|$ is small. By $\mathcal E_3$, the sum of the degrees in
$G[V_p]$ of the vertices in $S'$ is at least $(pd/2)\cdot|S'|$. On the
other hand, maximality of $W$ implies that every vertex of $S'$ has
fewer than $pd/8$ neighbors in $W$. Moreover, every subgraph of
$G[S']$ has minimum degree less than $pd/8$, and therefore
$e(G[S'])\leq (pd/8)\cdot|S'|$. Since these internal edges contribute twice
to the degree sum over $S'$, the total contributions from $W$, from
$S'$ itself, and from $S$ give
$(pd/2)\cdot|S'| \leq \sum_{v\in S'}d_{G[V_p]}(v)\leq(3pd/8)\cdot|S'|+2pd|S|$.
This yields $|S'|\leq 16|S|$.
Set $U'=U\cap W$ and note that almost all of the original bad set $U$ lies in
the dense part $W$. Indeed, since $\varepsilon\leq 1/160$,
$|U'|\geq k-16|S|\geq 9k/10$. Moreover, there are no edges between
$U'$ and $W\setminus U'$, since every neighbor of $U$ in
$V_p\setminus U$ belongs to $S$. 

We now thin $V_p$ further to get $V_{p/2}$ and $V_{p/4}$, showing that in this process, 
many vertices of $U'$ survive in $V_{p/4}$, only
slightly more survive in $V_{p/2}$, and every surviving component
inside $W$ remains large. Standard applications of the Chernoff
bound give that, with probability at least $1/2$,
$|U'\cap V_{p/2}|\leq9k/17$,
$|U'\cap V_{p/4}|\geq k/5$, and every vertex in $W$ has at least
$pd/64$ neighbors in $W\cap V_{p/4}$. Independently,
$(S\cup S')\cap V_{p/2}=\varnothing$ with probability
$2^{-|S\cup S'|}\geq2^{-17|S|}
\geq2^{-17\varepsilon p\lambda k}$.

Suppose all these events occur. Since
$S\cup S'$ is disjoint from
$V_{p/2}$, we have $V_{p/4}\subseteq V_{p/2}\subseteq W$. It follows that
$\delta(G[V_{p/4}])\geq pd/64$ and, by $\mathcal E_3$, every
vertex of $G$ has at most $2pd$ neighbors in $V_{p/4}$.
Let $\mathcal C$ be the components of $G[V_{p/4}]$ meeting $U'$,
and let $\mathcal C'$ be the components of $G[V_{p/2}]$ meeting
$V(\mathcal C)$. Since there are no edges between $U'$ and
$W\setminus U'$, all these components remain inside $U'$. Hence
$|V(\mathcal C)|\geq k/5$ and
$|V(\mathcal C')|\leq9k/17$.

It remains only to verify that these components do not spread far
beyond the original seed set. Every vertex of $V(\mathcal C')$ lies
in $U'\subseteq W$, and hence has at least $pd/64$ neighbors in
$W\cap V_{p/4}$. Since there are no edges from $U'$ to
$W\setminus U'$, all these neighbors lie in
$U'\cap V_{p/4}\subseteq V(\mathcal C)$. Thus
$V(\mathcal C')\subseteq\varphi(V(\mathcal C))$. We have therefore
shown that, conditioned on this outcome of $V_p$, $\mathcal E_2$
occurs with probability at least
$2^{-1-17\varepsilon p\lambda k}$. Finally, the usual Chernoff bound gives
$\mathbb P(\mathcal E_3)\geq1-e^{-\Omega(pd)}$, while by our choice
of $k$ and the assumption of the lemma,
$\mathbb P(\mathcal E_{1,k})\geq
2^{-\varepsilon p^2\lambda d-\log_2n}$. Since $p^2\lambda d=\omega(\log n)$,
$\mathcal E_{1,k}$ and $\mathcal E_3$ occur simultaneously with
probability at least $2^{-2\varepsilon p^2\lambda d}$. Using
$k\geq pd/4$, we conclude that
$$
\mathbb P(\mathcal E_2)
\geq 2^{-1-17\varepsilon p\lambda k-2\varepsilon p^2\lambda d}
\geq 2^{-80\varepsilon p\lambda k}.
$$
\end{proof}

Next, we turn to the sprinkling part of the argument, which, together with Lemma~\ref{lem: expander thinning}, proves the expansion assertion of Theorem~\ref{thm: regular expander subsampling}.
\begin{lemma}\label{lem: expander sprinkling}
    For every sufficiently small $\varepsilon>0$ the following holds. Let $0<\lambda\leq1$, let $G$ be a $\lambda$-expander on $n$ vertices with all degrees between $d$ and $(1+o(1))d$, and suppose that $p^2\lambda d=\omega(\log n)$. For every $2\leq k\leq3pn/5$, the following hold with probability less than $2^{-80\varepsilon p\lambda k}$. Every vertex in $V(G)$ has at most $2pd$ neighbors in $V_{p/4}$ and $\delta(G[V_{p/4}])\geq pd/64$. Further, there exists a collection $\mathcal C$ of components of $G[V_{p/4}]$ with $|V(\mathcal C)|\geq k/5$ such that the following holds. Let $\mathcal C'$ be the components of $G[V_{p/2}]$ which meet $V(\mathcal C)$. Then, $|V(\mathcal C')|\leq9k/17$ and $V(\mathcal C')$ is contained in $\varphi(V(\mathcal C))$.
\end{lemma}
As in the previous sections the sprinkling step needs the following lemma which, in some sense, takes on the role of Menger's connectivity theorem.
\begin{lemma}\label{lem: matching in expander}
    Let $0<\lambda\leq1$, let $G$ be a $\lambda$-expander and let $S\subseteq V(G)$ with $|S|\leq99|V(G)|/100$. Then, there is a matching in $G$ between $S$ and $V(G)\setminus S$ of size at least $\lambda|S|/400$.
\end{lemma}
\begin{proof}
    Either $S$ or $V(G)\setminus S$ is of size at most $|V(G)|/2$, and let $X\in\{S,V(G)\setminus S\}$ be that set. Note that $|X|\geq|S|/100$. Let $M$ be a maximum matching between $X$ and $V(G)\setminus X$ and suppose towards a contradiction that $|M|<\lambda|S|/400$. Let $X'\subseteq X$ be the set of vertices not covered by $M$. Then, $|X'|\geq|X|/2$ and, thus, since $G$ is a $\lambda$-expander, $|N_G(X')|\geq\lambda|S|/200$. Since $M$ is a maximum matching, we have that $N_G(X')\subseteq V(M)$, so that $M$ contains at least $\lambda|S|/400$ edges, a contradiction.
\end{proof}
\begin{proof}[Proof of Lemma~\ref{lem: expander sprinkling}]
Reveal $V_{p/4}$ and condition on an arbitrary outcome satisfying that every vertex of
$G$ has at most $2pd$ neighbors in $V_{p/4}$ and
$\delta(G[V_{p/4}])\geq pd/64$. Let $\mathcal L$ be the set of
components of $G[V_{p/4}]$, which, by the degree condition, each
have more than $pd/64$ vertices. Fix a collection
$\mathcal C\subseteq\mathcal L$ with
$k/5\leq |V(\mathcal C)|\leq9k/17$. We will show that after sprinkling,
with probability $1-\exp(-\Omega(p\lambda k))$, either the components
meeting $V(\mathcal C)$ grow beyond size $9k/17$, or they reach a
vertex outside $\varphi(V(\mathcal C))$. Since
$\mathcal C$ contains at most $64k/(pd)$ components, there are at most
$n^{64k/(pd)}=\exp(64k\log n/(pd))$ possible choices for
$\mathcal C$. As $p^2\lambda d=\omega(\log n)$, this is
$\exp(o(p\lambda k))$, which will be negligible compared with the
sprinkling probability. We may also
assume that $k\geq pd/64$, since otherwise no such collection
$\mathcal C$ exists. Let $\mathcal C'$ be the collection of components
of $G[V_{p/2}]$ meeting $V(\mathcal C)$, and let
$\mathcal A_{\mathcal C}$ be the event that either
$|V(\mathcal C')|>9k/17$ or
$V(\mathcal C')\not\subseteq\varphi(V(\mathcal C))$. Recall that
$V_{p/2}=V_{p/4}\cup V_{p'}$, where $V_{p'}$ is independent of
$V_{p/4}$ and $p'=p/(4-p)\geq p/4$. There are two cases, according
to whether $\varphi(V(\mathcal C))$ is already very large.

Suppose first that $|\varphi(V(\mathcal C))|\geq99n/100$. Every vertex
of $\varphi(V(\mathcal C))\setminus V(\mathcal C)$ lies outside
$V_{p/4}$, since otherwise it would already lie in a component of
$G[V_{p/4}]$ belonging to $\mathcal C$. Thus every such vertex which
is sprinkled joins one of the components in $\mathcal C'$. Since
$|V(\mathcal C)|\leq9k/17\leq27pn/85$, the expected number of these
newly sprinkled vertices is at least
$p'(99n/100-27pn/85)\geq11pn/50$, where we used that $p'=p/(4-p)$. A Chernoff bound therefore shows
that, with probability $1-\exp(-\Omega(pn))$,
$|V(\mathcal C')|\geq k/5+pn/5\geq k/5+k/3>9k/17$. Thus, when
$\varphi(V(\mathcal C))$ is very large, sprinkling simply makes the
resulting component collection too large.

We may therefore assume that
$|\varphi(V(\mathcal C))|<99n/100$. In this case there are many
disjoint ways to escape $\varphi(V(\mathcal C))$. Since every vertex
of $V(\mathcal C)$ has degree at least $d$, while every vertex of $G$
has at most $2pd$ neighbors in $V(\mathcal C)$, counting incidences
gives $d|V(\mathcal C)|\leq2pd|\varphi(V(\mathcal C))|$, and hence
$|\varphi(V(\mathcal C))|\geq|V(\mathcal C)|/(2p)$. Applying
Lemma~\ref{lem: matching in expander} to
$\varphi(V(\mathcal C))$, we obtain a matching $M$ between
$\varphi(V(\mathcal C))$ and its complement of size at least
$\lambda|V(\mathcal C)|/(800p)=\Omega(\lambda k/p)$. If both
endpoints of some edge of $M$ belong to $V_{p'}$, then the endpoint
inside $\varphi(V(\mathcal C))$ becomes connected to
$V(\mathcal C)$, while the other endpoint lies outside
$\varphi(V(\mathcal C))$, so
$V(\mathcal C')\not\subseteq\varphi(V(\mathcal C))$. Since the edges
of $M$ are disjoint, these events are independent, and therefore the
probability that no edge of $M$ has both endpoints in $V_{p'}$ is at
most $(1-p'^2)^{|M|}\leq\exp(-\Omega(p\lambda k))$.

Thus, in either case,
$\mathbb P(\mathcal A_{\mathcal C}\mid V_{p/4})
\geq1-\exp(-\Omega(p\lambda k))$. The number of possible choices of
$\mathcal C$ contributes only $\exp(o(p\lambda k))$, so a union
bound gives a conditional failure probability at most
$\exp(-cp\lambda k)$ for some absolute constant $c>0$. Choosing $\varepsilon>0$ sufficiently small proves the lemma.
\end{proof}

\begin{proof}[Proof of Theorem~\ref{thm: regular expander subsampling}]
    Fix a sufficiently small $\varepsilon>0$. If the expansion assertion failed, Lemma~\ref{lem: expander thinning} would give some $k$ for which the obstruction has probability at least $2^{-80\varepsilon p\lambda k}$, contradicting Lemma~\ref{lem: expander sprinkling}. 
\end{proof}

\begin{corollary}\label{cor: expander to tough subsampling}
    The following holds for every absolute constant $\varepsilon>0$. Let $\lambda\leq 1$, let $G$ be a $\lambda$-expander with minimum degree $d$ and maximum degree $(1+o(1))d$. Let $p\in[0,1]$ be such that $p^2\lambda d=\omega(\log n)$. Then, $G[V_p]$ is $(1-\varepsilon)$-tough with probability at least $1-2^{-\Omega(p^2\lambda d)}$.
\end{corollary}

\begin{proof}
    Without loss of generality, let us assume that $\varepsilon\leq 1/10$. By Theorem~\ref{thm: regular expander subsampling}, $G[V_p]$ is an $\Omega(p\lambda)$-expander with probability at least $1-2^{-\Omega(p^2\lambda d)}$. Further, a Chernoff bound and a union bound over all vertices show all degrees of $G[V_p]$ are $(1\pm \varepsilon/4)pd$ with probability at least $1-2^{-\Omega(pd)}=1-2^{-\Omega(p^2\lambda d)}$. Therefore, these events happen simultaneously with probability at least $1-2^{-\Omega(p^2\lambda d)}$. Suppose this is the case. We show that $G[V_p]$ must also be $(1-\varepsilon)$-tough.

Denote the minimum and maximum degree of $G[V_p]$ by $\delta$ and $\Delta$ and note that $\Delta/\delta\leq 1+3\varepsilon/4$, where we use that $\varepsilon\leq 1/10$. We also have $p\lambda\cdot\delta=\Omega(p^2\lambda d)=\omega(1)$. Suppose a set $S\subseteq V_p$ disconnects $G[V_p]$. Either $|S| \geq \delta/2$ or the smallest component has size at least $\delta-|S| \geq \delta/2$. Then expansion gives $|S|\geq\Omega(p\lambda\cdot \delta)$, and hence $|S|=\omega(1)$ in both cases. Every component of size $x\leq\delta$ sends at least $x(\delta-x+1)\geq\delta$ edges to $S$. Thus, there are at most $\Delta|S|/\delta=(1+3\varepsilon/4)|S|$ such components.

Suppose that there are $r$ components of size larger than $\delta$ and let $U$ be the union of the smallest $r/2$ of them. Since $\delta (r-1)/2\leq |U|\leq |V_p|/2$, and external neighbours of $U$ lie in $S$, expansion gives $|S|\geq\Omega(p\lambda|U|)=\Omega(p\lambda\cdot\delta(r-1))$, implying $r=o(|S|)$, where we use that $|S|=\omega(1)$. Consequently, $G[V_p]-S$ has at most $(1+3\varepsilon/4+o(1))|S|$ components, proving that $G[V_p]$ is $(1-\varepsilon)$-tough.
\end{proof}

\section{Well-connected bipartite subgraphs}
In this section we give another application of the thinning--sprinkling technique and use it to prove Theorem~\ref{thm: bipartite}, which says that every $k$-connected graph with $k=\omega(\log n)$ contains a spanning bipartite subgraph that is $\Omega(k)$-connected. Let $G$ be a $k$-connected graph, let $A\subseteq V(G)$ and $V_{1/2}$ be two independent $1/2$-random subsets of $G$ and set $B=V(G)\setminus A$. We write
$H=G[A\cap V_{1/2},B\cap V_{1/2}]$ for the corresponding random bipartite graph. 

The proof follows the same general pattern as in previous sections. The first lemma is the thinning step. It says that if $G$ has no highly connected spanning bipartite subgraph, then, with non-negligible probability, the random graph $H$ either has several components or has a component which is connected to some vertex of $G$ only by few edges.

\begin{lemma}\label{lem: bipartite thinning}
Let $\varepsilon>0$, let $G$ be a $k$-connected graph and suppose that $\varepsilon k\geq 4$. If $G$ contains no spanning bipartite subgraph which is $\varepsilon k/2$-connected, then, with probability at least $2^{-2\varepsilon k}$, $H$ either has at least three connected components, or has a connected component $C$ such that some vertex of $G$ has fewer than $\varepsilon k$ neighbors in $C$.
\end{lemma}

\begin{proof}
Condition on the bipartition $A\cup B=V(G)$, and write $G'=G[A,B]$. By assumption, $G'$ is not $\varepsilon k/2$-connected, so let $S\subseteq V(G)$ with $|S|<\varepsilon k/2$ be a cut-set of $G'$, and let $C$ be a largest component of $G'-S$.

Suppose first that $G'-S$ has exactly two components, that $C$ is $\varepsilon k/2$-connected, and that every vertex of $G$ has at least $\varepsilon k$ neighbors in $C$. In this case, keep the bipartition of $C$ inherited from $A\cup B$, and place each vertex outside $C$ on the side on which it has fewer neighbors in $C$. Every such vertex then has at least $\varepsilon k/2$ neighbors across the bipartition in $C$. Since adding a vertex with at least $r$ neighbors to an $r$-connected graph preserves $r$-connectivity, this gives a spanning bipartite $\varepsilon k/2$-connected subgraph of $G$, a contradiction.

Thus, at least one of the following three possibilities occurs. First, $G'-S$ has at least three components. In this case, if $V_{1/2}$ avoids $S$ and contains one fixed vertex from each of three components, then $H$ has at least three components. This happens with probability at least $2^{-|S|-3}\geq 2^{-2\varepsilon k}$. Second, some vertex $x\in V(G)$ has fewer than $\varepsilon k$ neighbors in $C$. Fix $u\in C$. If $V_{1/2}$ avoids $S$ and contains $u$, then the component of $u$ in $H$ is contained in $C$, and hence $x$ has fewer than $\varepsilon k$ neighbors in that component. This happens with probability at least $2^{-|S|-1}\geq 2^{-2\varepsilon k}$. Finally, $G'-S$ has exactly two components but $C$ is not $\varepsilon k/2$-connected. Since we may assume that every vertex has at least $\varepsilon k$ neighbors in $C$, it follows that $|C|\geq \varepsilon k$. Therefore, we can choose a cut-set $S'\subseteq C$ with $|S'|<\varepsilon k/2$ and vertices $u,v$ in distinct components of $C-S'$, together with a vertex $w$ in the other component of $G'-S$. If $V_{1/2}$ avoids $S\cup S'$ and contains $u,v,w$, then $H$ has at least three components. This happens with probability at least $2^{-|S|-|S'|-3}\geq 2^{-2\varepsilon k}$. The lemma follows.
\end{proof}

We next show that the obstruction produced by Lemma~\ref{lem: bipartite thinning} is exponentially unlikely. Together with the thinning lemma, this will immediately yield the desired spanning bipartite subgraph.

\begin{lemma}\label{lem: bipartite sprinkling}
There exists $\varepsilon>0$ such that the following holds whenever $k=\omega(\log n)$. With probability larger than $1-2^{-2\varepsilon k}$, $H$ has at most two connected components and every vertex of $G$ has at least $\varepsilon k$ neighbors in each component of $H$.
\end{lemma}

\begin{proof}
The main idea of the proof is to construct a $1/4$-random set $W$ so that the component of $v$ in $G[W]$ is contained in the component of $v$ in $H$, and then apply Theorem~\ref{thm: subsampling k-connected} to the set $W$. This will imply that every connected component of $H$ is large, and hence, by a simple counting argument, that $H$ has at most two connected components.

Fix $v\in V(G)$ and condition on $v\in V_{1/2}$. We construct the component of $v$ in $G[W]$ by a breadth-first exploration starting from $v$. Whenever the exploration reaches a new vertex $x$ from a vertex $y$ already in the explored component, we put $x$ in $W$ precisely when $x\in V_{1/2}$ and $x$ and $y$ lie on opposite sides of the random bipartition $A \cup B$. Since the predecessor edges form a tree, the corresponding color-change events are independent, and each newly exposed vertex is included with probability $1/2\cdot1/2=1/4$. Thus, apart from the harmless conditioning at $v$, this has the distribution of the component of $v$ in a $1/4$-random set $W$. After the exploration terminates, the membership of the remaining vertices in $W$ may be completed independently. Moreover, every predecessor edge used in the exploration belongs to $H$, so the component of $v$ in $G[W]$ is contained in the component of $v$ in $H$.

Applying Theorem~\ref{thm: subsampling k-connected} with sampling probability $1/4$, we obtain that $G[W]$ is connected with probability at least $1-2^{-c_0k}$ for some absolute constant $c_0>0$. Since $G$ is $k$-connected, every vertex has degree at least $k$. Since $k=\omega(\log n)$, by Chernoff bounds, for some absolute constant $c_1>0$, with probability at least $1-e^{-c_1k}$ we have $|W|>n/5$ and every vertex of $G$ has at least $k/8$ neighbors in $W$. It follows that, except with probability $e^{-c_2k}$ and as long as $\varepsilon\leq 1/8$, the component of $v$ in $H$ has more than $n/5$ vertices and every vertex of $G$ has at least $\varepsilon k$ neighbors in it, where $c_2>0$ is absolute.

A union bound over all $v\in V(G)$, together with the Chernoff bound $|V_{1/2}|<3n/5$ with probability $1-e^{-\Omega(n)}$, shows that for some absolute $c_3>0$, with probability at least $1-2^{-c_3 k}$ every component of $H$ has more than $n/5$ vertices and every vertex of $G$ has at least $\varepsilon k$ neighbors in each component. In particular, $H$ has at most two components. Finally, by decreasing $\varepsilon$ if necessary, we may assume that $2\varepsilon<c_3$.
\end{proof}

We can now finish the argument. If $G$ had no spanning bipartite $\varepsilon k/2$-connected subgraph, Lemma~\ref{lem: bipartite thinning} would imply that the obstruction in Lemma~\ref{lem: bipartite sprinkling} occurs with probability at least $2^{-2\varepsilon k}$, whereas Lemma~\ref{lem: bipartite sprinkling} bounds this probability strictly from above by $2^{-2\varepsilon k}$. This is a contradiction. Thus $G$ contains a spanning bipartite subgraph whose connectivity is a positive constant fraction of $k$, as claimed.

\section{Almost spanning cycles}
In this section we demonstrate how to use Theorems~\ref{thm: subsample tough}, \ref{thm: vertex-transitive subsampling} and Corollary~\ref{cor: expander to tough subsampling}. One key ingredient we use in conjunction with toughness is the following old result of Gallai.

\begin{proposition}[\cite{gallai64}]\label{thm: A-path}
    Let $G$ be a graph and $A\subseteq V(G)$. An $A$-path is a path of length at least $1$ with both endpoints in $A$. Then, $G$ either contains $k$ vertex-disjoint $A$-paths or there exists $Z\subseteq V(G)$ with $|Z|\leq 2k-2$ such that $G-Z$ contains no $A$-path.
\end{proposition}
This readily translates into a very useful connecting lemma for tough graphs.
\begin{lemma}\label{lem: A-connecting}
    Let $G$ be a $t$-tough graph and $A\subseteq V(G)$ with $|A|\geq 2$. Then, $G$ contains at least $\min\{1,t\}\cdot |A|/4$ vertex-disjoint $A$-paths.
\end{lemma}
\begin{proof}
    Set $t'=\min\{1,t\}$.
    Suppose that $G$ does not contain $t'|A|/4$ vertex-disjoint $A$-paths. By Proposition~\ref{thm: A-path}, there exists $Z\subseteq V(G)$ with $|Z|< t'|A|/2$ such that $G-Z$ contains no $A$-path. It follows that $G-Z$ has at least $|A\setminus Z|>|A|-t'|A|/2\geq |A|/2$ components. Since $G$ is $t'$-tough, it follows that $|Z|\geq t'\cdot |A|/2$, contradicting our earlier observation that $|Z|<t'|A|/2$.
\end{proof}
Using this, we obtain the following key tool for creating long cycles.
\begin{lemma}\label{lem: long cycle}
    Let $G$ be a $t$-tough graph on $n$ vertices for some $0<t\leq 1$ and set $k=\log_{1-t/4} (1/n) +1$. Let $V_0,\ldots,V_{k}$ be disjoint subsets of $V(G)$ such that $G[V_0]$ contains a spanning linear forest $F$ and, for $1\leq i\leq k$, $G[V_i]$ is $t$-tough and $G$ contains a matching between the endpoints of $F$ and $V_i$ covering all endpoints of $F$. Then, $G$ contains a cycle spanning $V_0$.
\end{lemma}
\begin{proof}
    We prove by induction that $G[V_0\cup\ldots\cup V_i]$ contains a linear forest $F_i$ containing $V_0$ whose endpoints are also endpoints of $F$ with at most $(1-t/4)^in$ components. The base case holds by choosing $F_0=F$. Suppose we have proven the statement up to $i-1$ and let $F_{i-1}$ be such a linear forest. If $F_{i-1}$ is already a path, set $F_i=F_{i-1}$. Otherwise, for each path in $F_{i-1}$ let us choose one endpoint arbitrarily. We denote this set by $W$ and note that $F_{i-1}$ contains exactly $|W|$ components, so $|W|\leq (1-t/4)^{i-1}n$ by the induction hypothesis. Since the endpoints of $F_{i-1}$ are also endpoints of $F$, there is a matching $M$ in $G$ between $W$ and $V_i$ covering $W$. Let $A$ be the set of endpoints of edges of $M$ which are in $V_i$. By Lemma~\ref{lem: A-connecting}, $G[V_i]$ contains a linear forest $T$ of $t|A|/4=t|W|/4$ vertex-disjoint $A$-paths. Combining these paths with edges of $M$, we obtain a linear forest $F_i$ in which $t|W|/4$ pairs of paths from $F_{i-1}$ are merged. Therefore $F_i$ is a linear forest with at most $\max\{1,(1-t/4)|W|\}\leq \max\{1,(1-t/4)^{i}n\}$ components, whose endpoints are a subset of endpoints of $F_{i-1}$.

    Hence, $F_{k-1}$ is a path spanning $V_0$ whose endpoints are endpoints of $F$. Therefore, $G$ contains a matching into $V_{k}$ covering these two endpoints. Since $G[V_k]$ is connected, there exists a cycle in $G$ which contains $F_{k-1}$ and hence, $V_0$ as well.
\end{proof}

\subsection{Long cycles in tough graphs}
\begin{proof}[Proof of Theorem~\ref{thm: tough cycle}]
    We show that for every absolute constant $0<\varepsilon\leq 1/2$, $G$ contains a cycle of length at least $(1-2\varepsilon)n$. Set $k=\log_{3/4} (1/n)+1$ and $p=\varepsilon/k$. Independently for each vertex, assign it to one of $V_0,\ldots,V_k$, choosing $V_0$ with probability $(1-\varepsilon)$ and, for $1\leq i\leq k$, $V_i$ with probability $\varepsilon/k$. 
    
    We keep track of several events. For every $1\leq i\leq k$, let $\mathcal A_i$ be the event that $G[V_i]$ is $1$-tough. Since $p^2t=\omega(\log n)$, by Theorem~\ref{thm: subsample tough} and a union bound,
    these events hold simultaneously with probability at least $1-n\cdot 2^{-\omega(\log n)}=1-o(1)$. Next up, for any independent set $I\subseteq V(G)$ of $G$ and for $1\leq i\leq k$, let $\mathcal B_{I,i}$ be the event that $V_i$ contains at least $2|I|$ neighbors of $I$. Since $G$ is $t$-tough, we have that $|N(I)|\geq t|I|$, unless $G$ is the complete graph in which case the statement follows trivially. Therefore, by Chernoff's bound, we get that $\mathcal B_{I,i}$ happens with probability at least $1-e^{-\Omega(t|I|/\log n)}=1-e^{-\omega(\log n\cdot |I|)}$. But the number of choices of $(I,i)$ with $I$ of size $\ell$, for any $\ell$, is at most $n^\ell k\leq e^{O(\ell\log n)}$, so that by a union bound, $\mathcal B_{I,i}$ holds with high probability for all choices of $I$ and $i$. Finally, let $\mathcal C$ be the event that $V_0$ contains at least $(1-2\varepsilon)n$ vertices, which again happens with high probability by a Chernoff bound. Therefore, we have that all $\mathcal A_i$, $\mathcal B_{I,i}$ and $\mathcal C$ happen simultaneously with probability $1-o(1)$. For the remainder of the argument, let us assume that $V_0,V_1,\ldots, V_{k}$ is a partition for which all these events are satisfied.

    Let $F$ be a spanning linear forest of $G[V_0]$ minimizing its number of components. Let $X$ be the set of endpoints of $F$. By minimality of $F$, $G[X]$ can contain only the matching of edges connecting the pair of endpoints of the same paths of $F$. Thus for every $U\subseteq X$, there exists an independent set $I\subseteq U$ of at least half the size of $U$. For every $i$, it follows from $\mathcal B_{I,i}$ that $U$ has at least $|U|$ neighbors in $V_i$. As this is exactly Hall's condition, $G$ contains a matching between $X$ and $V_i$ which covers $X$, for every $i$. Hence, by Lemma~\ref{lem: long cycle}, $G$ contains a cycle which spans $V_0$. It follows from $\mathcal C$ that this cycle has length at least $(1-2\varepsilon)n$.
\end{proof}
\subsection{Long cycles in robustly tough nearly regular graphs}
In this section we prove the following theorem together with its applications to vertex-transitive and expander graphs.
\begin{theorem}\label{thm: weak meta theorem}
    Let $G$ be a graph on $n$ vertices with minimum degree
    $d=\omega(\log^3 n\,\log\log n)$ and maximum degree
    $(1+o(1))d$. Suppose that $G[V_p]$ is $1/2$-tough
    with probability at least $1-n^{-2}$ whenever
    $p=\Omega(1/\log n)$. Then $G$ contains a cycle of length
    $(1-o(1))n$.
\end{theorem}

To prove Theorem~\ref{thm: weak meta theorem}, we first establish the
following useful auxiliary lemma. If every vertex has approximately the same degree into
the preceding and following parts, and these degrees are nearly equal
across vertices, it gives a collection of disjoint paths covering
almost all vertices, with all endpoints in the first and last parts.

\begin{lemma}\label{lem: paths through layers}
    Let $G$ be a graph on $n$ vertices with a partition
    $W_1,\ldots,W_m$, where $m\geq 2$ and $W_1$ is a largest
    part. Let $d>0$ and $0\leq\alpha\leq1/2$.
    Suppose that there are numbers $d_v\in[d,(1+o(1))d]$,
    $v\in V(G)$, such that, for every $1\leq i<m$,
    every vertex $v$ of $G[W_i,W_{i+1}]$ has degree
    $(1\pm\alpha)d_v$. Then $G$ contains a collection of
    vertex-disjoint paths with endpoints in $W_1\cup W_m$
    which cover at least $(1-2\alpha m-o(1))n$ vertices.
\end{lemma}
\begin{proof}
    Discard all edges which do not join consecutive parts, and
    set $\rho=(1+\alpha)/(1-\alpha)$. For $U\subseteq V(G)$,
    write $D(U)=\sum_{v\in U}d_v$. Let $Z$ be any set meeting
    every path from $W_1$ to $W_m$, and put $Z_i=Z\cap W_i$.
    Let $R_i$ consist of the vertices of $W_i$ reachable from
    $W_1\setminus Z$ in $G-Z$. Then $R_1=W_1\setminus Z_1$
    and $R_m=\varnothing$. Every edge from $R_i$ to $W_{i+1}$
    ends in $R_{i+1}\cup Z_{i+1}$, so counting these edges gives
    $D(R_i)\leq\rho\bigl(D(R_{i+1})+D(Z_{i+1})\bigr)$.
    Iterating this inequality yields
    \[
        D(W_1)\leq\sum_{i=1}^m\rho^{i-1}D(Z_i)
        \leq\rho^{m-1}(1+o(1))d|Z|.
    \]
    Since $W_1$ is largest, $D(W_1)\geq d|W_1|\geq dn/m$.
    Thus every such $Z$ has at least
    $(1-o(1))\rho^{-m+1}n/m\geq(1-2\alpha(m-1)-o(1))n/m$ vertices, where we used $\rho^{-1}\geq 1-2\alpha$. By Menger's theorem,
    there are at least this many vertex-disjoint paths from
    $W_1$ to $W_m$. Each contains at least $m$ vertices, so
    together they cover at least $(1-2\alpha(m-1)-o(1))n$
    vertices.
\end{proof}
\begin{proof}[Proof of Theorem~\ref{thm: weak meta theorem}]
    We show that for every absolute constant $0<\varepsilon<1/4$,
    $G$ contains a cycle of length at least $(1-2\varepsilon)n$.
    Set $\gamma=\varepsilon/10$, $k=\log_{7/8}(1/n)+1$,
    $p=\varepsilon/k$ and $m=k/\gamma$.

    Independently for each vertex, assign it to one of
    $V_0,\ldots,V_k$, choosing each $V_i$, $1\leq i\leq k$,
    with probability $p$ and $V_0$ with probability
    $1-\varepsilon$. Since $p=\Theta(1/\log n)$, our assumption
    and a union bound imply that all $G[V_i]$, $1\leq i\leq k$,
    are $1/2$-tough with probability $1-o(1)$. Since $pd=\omega (\log n)$,
    Chernoff's bound and a union bound show that, with
    probability $1-o(1)$, every vertex $v\in V(G)$ has at least
    $(1-o(1))pd$ neighbors in each $V_i$ and
    $(1\pm O(\sqrt{\log n/d}))(1-\varepsilon)d_G(v)$ neighbors in
    $V_0$. Another application of Chernoff's bound gives
    $|V_0|=(1-\varepsilon+o(1))n$ with high probability.
    Fix a partition satisfying all these properties.

    We next partition $V_0$ into $W_1,\ldots,W_m$. We require
    the degrees into these parts to be controlled for every
    vertex of $G$, including vertices outside $V_0$.
    Put, for every $v\in V(G)$, $d_v=(1-\varepsilon)d_G(v)/m=(1\pm o(1))(1-\varepsilon)d/m$ and
    $\alpha=C\sqrt{m\log d/d}$, where $C$ is a sufficiently
    large absolute constant. We claim that there is a partition
    satisfying
    \[
        d_G(v,W_j)=(1\pm\alpha)d_v
        \qquad\text{for every $v\in V(G)$ and $1\leq j\leq m$.}
    \]
    To see this, assign each vertex of $V_0$ independently and
    uniformly to one of the $m$ parts. For every $v\in V(G)$,
    the expected degree into $W_j$ is
    $d_G(v,V_0)/m=(1\pm O(\sqrt{\log n/d}))d_v$.
    The relative error here is $o(\alpha)$, since
    $\sqrt{\log n/d}/\alpha
    =O(\sqrt{\log n/(m\log d)})=o(1)$.
    Moreover, $\alpha^2d_v\geq C^2(1-\varepsilon)\log d$.
    Chernoff's bound therefore shows that, for sufficiently
    large $C$, the probability of violating the desired
    estimate for a fixed pair $(v,j)$ is at most $d^{-4}$.
    Each such bad event depends only on the assignments of
    the vertices in $N_G(v)\cap V_0$. Two events can therefore
    only depend on each other if their vertices have a common
    neighbor in $V_0$. There are at most $(1+o(1))d^2$ such vertices,
    each giving $m$ events, so every event is independent of
    all but at most $(1+o(1))md^2$ others. Since
    $d^{-4}(md^2+1)=o(1)$, the Lov\'asz Local Lemma gives the
    required partition. Relabel the parts so that $W_1$ is largest.

    The assumption that $d=\omega(\log^3 n\,\log\log n)$ together with $m = O(\log n)$ implies $m^3\log d/d=
    o(1)$. Hence $m\alpha=C\sqrt{m^3\log d/d}=o(1)$.
    Applying Lemma~\ref{lem: paths through layers} to $G[V_0]$, we obtain
    a linear forest $F$ covering at least
    $(1-2\alpha m-o(1))|V_0|=(1-o(1))|V_0|$
    vertices of $V_0$, whose paths have their endpoints in $W_1\cup W_m$.
    Let $X$ be the set of endpoints of $F$. Since
    $X\subseteq W_1\cup W_m$ and our degree estimates hold
    for every vertex of $G$, every vertex of each $V_i$
    has at most $2(1+\alpha)(1+o(1))d/m\leq 3d/m=3\gamma d/k$
    neighbors in $X$. Here we used $\alpha=o(1)$.
    On the other hand, every vertex of $X$ has at least
    $(1-o(1))pd\geq\varepsilon d/(2k)$ neighbors in each
    $V_i$.
    For any $A\subseteq X$, counting edges from $A$ to $V_i$
    therefore gives
    $|N_G(A)\cap V_i|\geq \varepsilon|A|/(6\gamma)\geq |A|$,
    where the last inequality follows from
    $\gamma=\varepsilon/10$.
    Thus Hall's condition holds, and for every $1\leq i\leq k$
    there is a matching from $X$ into $V_i$ covering $X$.

    We can now apply Lemma~\ref{lem: long cycle} to $V(F),V_1,\ldots,V_k$, with $t=1/2$. We obtain a cycle containing $V(F)$, and hence of length at least $(1-\varepsilon-o(1))n\geq(1-2\varepsilon)n$, as required.
\end{proof}

    \begin{proof}[Proof of Theorem~\ref{thm: vertex-transitive cycle} and Theorem~\ref{thm: expander cycle}]
        Under the given assumptions, we have that if $p=\Omega(1/\log n)$ then $p^2d=\omega(\log n)$ for Theorem~\ref{thm: vertex-transitive cycle} and $p^2\lambda d=\omega(\log n)$ for Theorem~\ref{thm: expander cycle}. Therefore we can use Theorem~\ref{thm: vertex-transitive subsampling} and Corollary \ref{cor: expander to tough subsampling} to deduce that $G[V_p]$ is $1/2$-tough with probability at least $1-2^{-\omega(\log n)} \geq 1-n^{-2}$. Hence the results follow from Theorem~\ref{thm: weak meta theorem}.
\end{proof}
We note that the $\log\log n$ term in Theorem~\ref{thm: weak meta theorem} and, hence, also in Theorems~\ref{thm: vertex-transitive cycle} and~\ref{thm: expander cycle} can be removed by selecting the linear forest more efficiently. This can be achieved by using the methods from \cite{lineararboricity} or \cite{montgomery}. We choose not to do this to keep the argument simple and self-contained.
\section{Concluding remarks}
Combining the methods of this paper with absorption, we have also proved that connected vertex-transitive graphs of polylogarithmic degree are Hamiltonian whenever they contain a short odd cycle or have small diameter. In particular, there exists $\varepsilon>0$ such that every connected $\Omega(n^{1-\varepsilon})$-regular vertex-transitive graph is Hamiltonian, extending the result of Bedert, Dragani\'c, M\"uyesser, and Pavez-Sign\'e~\cite{bedert-draganic-muyesser-pavez-signe} from Cayley graphs to vertex-transitive graphs. As these results require additional absorption arguments beyond the scope of this paper, we defer their proofs to subsequent work.


\vspace{0.25cm}
\noindent
{\bf Acknowledgments.}
The authors would like to thank Michael Krivelevich and David Munh\'a Correia for fruitful discussions on long cycles in tough graphs.

\vspace{0.25cm}
\noindent
{\bf Statement of AI use.}
All mathematical results and proofs in this paper were developed by the authors independently of generative AI. ChatGPT models were used only to improve the language and presentation of the manuscript.

\end{document}